%% file: double_hyperbolic_arxiv.tex
\documentclass[11pt, oneside]{article}

\usepackage{enumerate}
\usepackage{amscd}
\usepackage{amsmath}
\usepackage{epsfig}
\usepackage{amssymb}
\usepackage{amsthm}    
\usepackage{latexsym}
\usepackage{color} 
\usepackage{graphicx}
\usepackage{url}

\numberwithin{equation}{section}

\theoremstyle{plain}
\newtheorem{theorem}{Theorem}[section]
\newtheorem{proposition}[theorem]{Proposition} 
\newtheorem{lemma}[theorem]{Lemma} 
\newtheorem{corollary}[theorem]{Corollary} 
\theoremstyle{definition}
\newtheorem{definition}[theorem]{Definition} 
\newtheorem{question}[theorem]{Question} 
\theoremstyle{remark}
\newtheorem{remark}[theorem]{Remark} 
\newtheorem{example}[theorem]{Example} 
{\it}{\rm}  

\newcommand{\sss}{\scriptscriptstyle}
\newcommand{\SDF}{SDF} 
\newcommand{\email}[1]{{\scriptsize{\it E-mail address}\/: {\rm #1}} }

\makeatletter
\def\blfootnote{\gdef\@thefnmark{}\@footnotetext}
\makeatother

\begin{document}
\blfootnote{\textup{2020} \textit{Mathematics Subject Classification}.
51M10, 51F99, 51M25, 52B11}

\blfootnote{\textit{key words and phrases}.
double hyperbolic space, polytope, boundary at infinity, Schl\"{a}fli differential formula (\SDF{}).}

\blfootnote{\email{lizhaozhang@alum.mit.edu}}

\begin{titlepage}
\title{On the total volume of the double hyperbolic space}

\author{Lizhao Zhang
}
\date{}
\end{titlepage}

\maketitle

\begin{abstract}
Let the \emph{double hyperbolic space} $\mathbb{DH}^n$, 
proposed in this paper as an extension of the hyperbolic space $\mathbb{H}^n$,
contain a two-sheeted hyperboloid with  the two sheets
connected to each other along the boundary at infinity.
We propose to extend the volume of convex polytopes in $\mathbb{H}^n$ 
to polytopes in $\mathbb{DH}^n$, where the volume is invariant under isometry
but can possibly be complex valued. 
We show that the total volume of $\mathbb{DH}^n$ is equal to 
$i^n V_n(\mathbb{S}^n)$ for both even and odd dimensions,
and prove a Schl\"{a}fli differential formula (\SDF{}) for $\mathbb{DH}^n$.
For $n$ odd, the volume of a polytope in $\mathbb{DH}^n$
is shown to be completely determined by its intersection with $\partial\mathbb{H}^n$
and induces a new intrinsic \emph{volume} on $\partial\mathbb{H}^n$
that is invariant under M\"{o}bius transformations.
\end{abstract}

\section{Introduction}
\label{section_intro}

In this paper we extend the hyperbolic space $\mathbb{H}^n$ to a space formed 
by a \emph{two}-sheeted hyperboloid 
by identifying their boundary at infinity projectively,
and extend the volume properly to the new space.
Denote the new space by \emph{double hyperbolic space} $\mathbb{DH}^n$,
and the other sheet by $\mathbb{H}^n_{-}$.
We remark that $\mathbb{H}^n_{-}$ is not isometric to $\mathbb{H}^n$ 
and the length element $ds$ on $\mathbb{H}^n_{-}$ is negative, 
with the associated volume element also negative for $n$ odd
(see Section~\ref{section_preliminaries} for more details).

\subsection{Background and motivation}

Among the models of $\mathbb{H}^n$, two of them are the
\emph{upper half-space model} and the \emph{hemisphere model}, and in both models
the boundary at $x_0=0$ corresponds to the \emph{boundary at infinity} $\partial\mathbb{H}^n$.
It seems natural to ask if some kind of theory can be developed across the boundary,
so that the upper half-space model can be extended to the lower half-space, and 
the hemisphere model to the lower hemisphere,
making them more like a ``full-space model'' or ``full-sphere model''
in some sense. 

In this note, we attempt to provide such a theory by adding $\mathbb{H}^n_{-}$ 
to the picture as the lower half-space or the lower hemisphere respectively
with a natural identification $\partial\mathbb{H}^n=\partial\mathbb{H}^n_{-}$ at $x_0=0$.
Notice that $\mathbb{DH}^n$ is homeomorphic to the standard unit $n$-sphere $\mathbb{S}^n$.

\subsection{Main results}

We first introduce some (new) basic notions of $\mathbb{DH}^n$ using the hyperboloid model,
but the notions can be easily extended to other models.
An \emph{isometry} of $\mathbb{DH}^n$ is an isometry of $\mathbb{H}^n$  
that also preserves the antipodal points (the \emph{counterpart}) in $\mathbb{H}^n_{-}$ still as antipodal points.
A \emph{half-space} in $\mathbb{DH}^n$ is obtained by gluing a half-space in $\mathbb{H}^n$ 
and its antipodal image in $\mathbb{H}^n_{-}$ along 
their common boundary on $\partial\mathbb{H}^n$ (after the identification, see Figure~\ref{figure_half_space}).
Note that by the definition, $\mathbb{H}^n$ and $\mathbb{H}^n_{-}$ are not half-spaces in $\mathbb{DH}^n$,
though both of them look like a ``half'' of $\mathbb{DH}^n$.
A \emph{polytope} in $\mathbb{DH}^n$ is a finite intersection of half-spaces in $\mathbb{DH}^n$,
the same as the union of a polytope in $\mathbb{H}^n$ (possibly unbounded) 
and its antipodal image in $\mathbb{H}^n_{-}$.
So a polytope in $\mathbb{DH}^n$ is always symmetric between $\mathbb{H}^n$ and $\mathbb{H}^n_{-}$ through antipodal points,
but by definition a polytope in $\mathbb{H}^n$ by itself is not a polytope in $\mathbb{DH}^n$.

\begin{figure}[h]
\centering
  \includegraphics[width=0.3\textwidth]{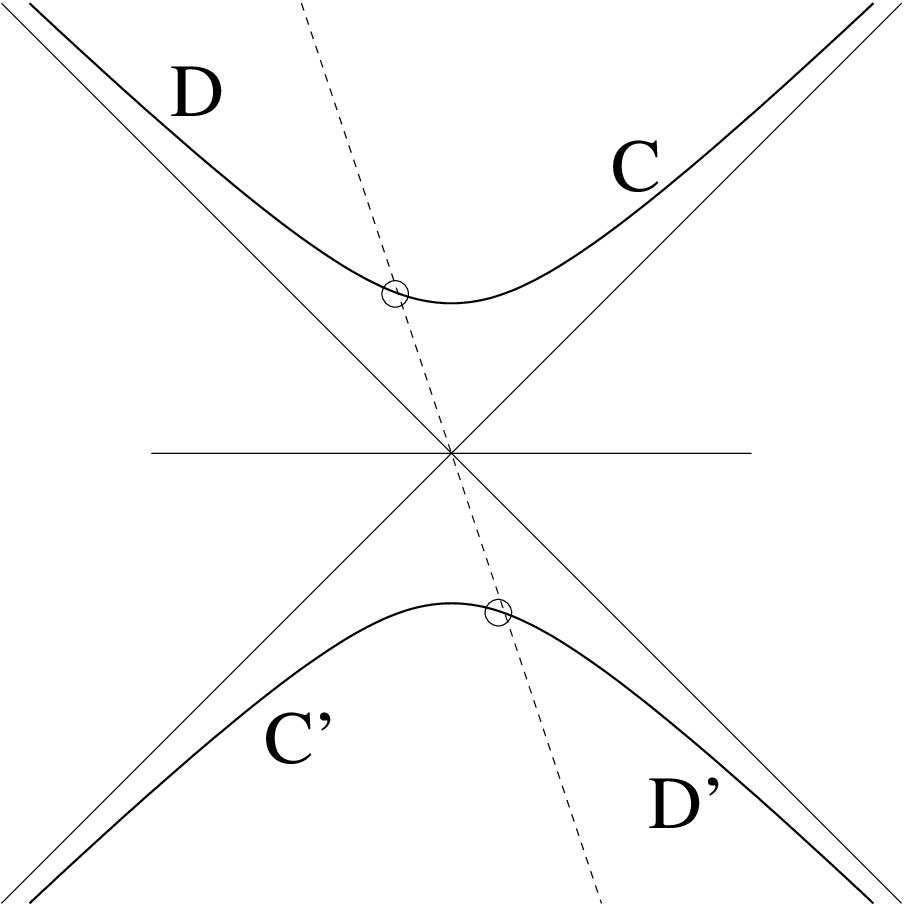}
\caption{A hyperplane cut $\mathbb{DH}^n$ into two half-spaces, the union of $C$ and $C'$, and the union of $D$ and $D'$}
\label{figure_half_space}
\end{figure}

Our focus is to extend the volume on $\mathbb{H}^n$
to $\mathbb{DH}^n$ in a proper way, such that it is compatible with
the volume elements of both $\mathbb{H}^n$ and $\mathbb{H}^n_{-}$,
and is also properly defined across $\partial\mathbb{H}^n$.
However, the integral of the volume element across $\partial\mathbb{H}^n$
cannot be defined by the standard Lebesgue integration.
To fix this issue, we treat the volume as the limit of the integral of 
a complex perturbation of the volume element (as a complex measure on $\mathbb{R}^n$). 
We shall note that in the complex perturbation the underlying space is still real, but endowed with 
complex valued ``Riemannian metrics'' instead, with no complex geometry involved.
Thus we only use complex analysis in a very limited way
mainly to compute the integral of complex valued functions in the real space.
In fact, except for some basic understanding of the Riemannian metrics 
on the different models of $\mathbb{H}^n$,
no prior knowledge of differential geometry is assumed of the reader.

For this newly introduced volume on $\mathbb{DH}^n$,
it is shown in Theorem~\ref{theorem_volume_invariant}
(whose proof constitutes an essential part of this paper)
that the volume of a polytope $P$ in $\mathbb{DH}^n$ (denote by $V_n(P)$)
is well defined and invariant under isometry.
This is true even when $P$ contains a non-trivial portion 
of $\partial\mathbb{H}^n$, but in this case the volume is only finitely (but not countably) additive
and may also be complex valued.
A central tool we use to develop the theory is the Schl\"{a}fli differential formula (\SDF{}),
a formula that applies to the volume of polytopes in 
the spherical, Euclidean, or hyperbolic space of constant curvature $\kappa$.
Some remarks on the history of the formula can be found in Milnor \cite{Milnor:Schlafli}.
See also Rivin and Schlenker~\cite{RivinSchlenker} and Su{\'a}rez-Peir{\'o}~\cite{Suarez:deSitter}
for some generalizations.

With a proper choice of the extension of the volume from $\mathbb{H}^n$ to $\mathbb{DH}^n$
(Definition~\ref{definition_volume_P}, though the choice is not unique),
we obtain the total volume of $\mathbb{DH}^n$ in the following.

\begin{theorem}
\label{theorem_total_volume}
The total volume of $\mathbb{DH}^n$ is 
\begin{equation}
V_n(\mathbb{DH}^n)=i^n V_n(\mathbb{S}^n)
\end{equation}
for both even and odd dimensions,
where $V_n(\mathbb{S}^n)$ is the $n$-dimensional volume of the standard 
unit $n$-sphere $\mathbb{S}^n$.
\end{theorem}

Heuristically, the two-sheeted hyperboloid may be thought of as
a ``sphere'' of radius $i=\sqrt{-1}$ in the hyperboloid model,
and therefore it suggests why the formula above might make sense.
While the choice of the value of $V_1(\mathbb{DH}^1)$ is not unique
as explained later in Section~\ref{section_definition_volume_P},
we provide more reasoning about the choice of $2\pi i$
in the sense of complex analysis.
We prove a \SDF{} for $\mathbb{DH}^n$.

\begin{theorem}
\label{theorem_schlafli}
\emph{(Schl\"{a}fli differential formula for $\mathbb{DH}^n$)}
For $n\ge 2$, let $P$ be a polytope in $\mathbb{DH}^n$ that does not contain any ideal vertices, 
and for each $(n-2)$-dimensional face $F$, let $\theta_F$ be the dihedral angle at $F$.
Then for $\kappa=-1$,
\begin{equation}
\kappa\cdot dV_n(P)=\frac{1}{n-1}\sum_{F}V_{n-2}(F)\,d\theta_F,
\end{equation}
where the sum is taken over all $(n-2)$-faces $F$ of $P$.
For $n-2=0$, $V_0(F)$ is the number of points in $F$.
\end{theorem}

If the intersection of $P$ and the faces of some of the closed half-spaces (that form $P$) 
is a point on $\partial\mathbb{H}^n$, we call it an \emph{ideal vertex}
(see Figure~\ref{figure_ideal} for an illustration using the upper half-space model).
The main reason for excluding ideal vertices in Theorem~\ref{theorem_schlafli}
is to simplify the proof.
To some extent, we were expecting a \SDF{} 
for $\mathbb{DH}^n$ first, and then used it as a guideline
to properly introduce a volume on polytopes in $\mathbb{DH}^n$.
We also have the following finiteness result.

\begin{theorem}
\label{theorem_volume_finite}
\emph{(Uniform boundedness of $V_n(P)$ for a fixed $m$)}
Let $P$ be a polytope in $\mathbb{DH}^n$
and be the intersection of at most $m$ half-spaces in $\mathbb{DH}^n$, then 
\begin{equation}
\label{equation_volume_finite}
|V_n(P)|\le\frac{m!}{2^{m-1}}V_n(\mathbb{S}^n).
\end{equation}
\end{theorem}

The bound given in (\ref{equation_volume_finite}) is very loose, 
but it serves our main purpose of showing that $V_n(P)$ is finite and 
uniformly bounded for a fixed $m$.

\begin{theorem}
\label{theorem_volume_real_imaginary}
Let $\mathcal{H}$ be the algebra over $\mathbb{DH}^n$ generated by half-spaces in $\mathbb{DH}^n$, and $P\in\mathcal{H}$.
Then $V_n(P)$ is real for $n$ even,
and $V_n(P)$ is imaginary for $n$ odd and is completely determined by $P\cap \partial\mathbb{H}^n$.
\end{theorem}

\begin{remark}
Assume $n$ is odd. Although a convex polytope in $\mathbb{H}^n$ always has real volume,
recall that by definition a convex polytope is not an element in $\mathcal{H}$,
so it does not contradict Theorem~\ref{theorem_volume_real_imaginary}.
Theorem~\ref{theorem_volume_real_imaginary}
means that the information of $V_n(P)$ is completely encoded 
in the boundary at infinity $\partial\mathbb{H}^n$, 
which is crucial for exploring new geometric properties
of $\partial\mathbb{H}^n$ on top of the standard conformal structure on a sphere.
In fact, assuming $V_n(P)$ is well defined (Theorem~\ref{theorem_volume_invariant}),
the interested reader may directly read the proof of Theorem~\ref{theorem_volume_real_imaginary}
without first going through the long proof of Theorem~\ref{theorem_volume_invariant},
as the two proofs are independent.
\end{remark}

An important application of Theorem~\ref{theorem_volume_real_imaginary} 
is that for $n$ odd and $P\in\mathcal{H}$,
it induces an intrinsic \emph{volume} on $\partial\mathbb{H}^n$ by simply assigning
$V_n(P)$ to $G$ where $G=P\cap \partial\mathbb{H}^n$
(will be adjusted by a constant factor later).
For $n=2m+1$, denote the volume of $G$ by $V_{\infty,2m}(G)$,
then $V_{\infty,2m}(G)$ is invariant under M\"{o}bius transformations (Theorem~\ref{theorem_invariant_mobius}).
To our knowledge, both the definition of $V_{\infty,2m}(G)$ and its (conformal) invariance property are new.
Thus on top of the standard conformal structure on $\partial\mathbb{H}^{2m+1}$ as a sphere,
we obtain $V_{\infty,2m}(G)$ as a new geometric invariant of $\partial\mathbb{H}^{2m+1}$.
We also show that $\partial\mathbb{H}^{2m+1}$ has hidden geometric properties 
of the spherical, (double) hyperbolic, and Euclidean spaces at the same time
(Theorem~\ref{theorem_polytope_volume_infinity}).
We remark that $V_{\infty,2m}(G)$ is not induced by any volume form 
on $\partial\mathbb{H}^{2m+1}$ as a differentiable manifold,
and it takes values positive, negative or zero as well.

If $P$ is also a polytope in $\mathbb{DH}^n$, 
namely, a finite intersection of closed half-spaces in  $\mathbb{DH}^n$, 
we call $G$ a \emph{polytope} in $\partial\mathbb{H}^n$.
We show that $V_{\infty,2m}(G)$ 
satisfies a new version of \SDF{} for $\partial\mathbb{H}^{2m+1}$
(Theorem~\ref{theorem_schlafli_boundary_infinity_new}).

\subsection{Related works}
\label{section_related_works}

In the hyperboloid model of $\mathbb{H}^n$,
though the hyperboloid in $\mathbb{R}^{n,1}$ has two sheets, 
usually most works deal with only one of the two sheets of the hyperboloid
or identify the two sheets projectively.
For some basic notions of hyperbolic space, 
see Cannon \emph{et al.}~\cite{Cannon:hyperbolic}
and Milnor~\cite{Milnor:hyperbolic}.
To extend the hyperbolic space $\mathbb{H}^n$ beyond the boundary at infinity
and have a well defined volume across the boundary,
one way is to use the de Sitter space as an extension.
In the Klein model of $\mathbb{H}^n$, 
which is the interior of an open disk in the projective space $\mathbb{RP}^n$
and has a metric
\[ds^2=\frac{dx_1^2+\cdots+dx_n^2}{1-x_1^2-\cdots-x_n^2}
+\frac{(x_1dx_1+\cdots+x_ndx_n)^2}{(1-x_1^2-\cdots-x_n^2)^2},
\]
Cho and Kim \cite{ChoKim} applied the same metric formula to the outside of the disk
and defined a complex valued volume on the extended space. 
In this interpretation, straight lines in the projective space
across the boundary $\partial\mathbb{H}^n$ can also be viewed as geodesics
in the extended space,
which in fact was the main motivation for the construction.
At the outside of the disk, the de Sitter part, it is shown in \cite{ChoKim} that 
its metric $ds^2$ is the negative of the metric $ds^2$ on the de Sitter space.
Under a similar setting, the basic notions such as length and angle 
were also explored through cross ratio (see Schlenker \cite{Schlenker:cross}).

More precisely, the extended space Cho and Kim \cite{ChoKim} considered is a double covering
of the projective space and homeomorphic to the standard $n$-sphere $\mathbb{S}^n$.
It can be viewed as obtained from a radial projection from the origin of $\mathbb{R}^{n,1}$ 
that maps the two-sheeted hyperboloid and the de Sitter space to the unit sphere
$x_0^2+\cdots+x_n^2=1$,
and then changing the induced metric on the de Sitter part from $ds^2$ to $-ds^2$.
Namely, changing the spacelike geodesics
to timelike geodesics, and vice versa.
In this model, the extended space contains three open parts,
with one de Sitter part (for $n=1$ it has two connected components)
and two hyperbolic parts that both are isometric to $\mathbb{H}^n$.
In our construction of $\mathbb{DH}^n$,
the lower sheet $\mathbb{H}^n_{-}$ is \emph{not} isometric to $\mathbb{H}^n$.

While the complex valued geometry Cho and Kim constructed is consistent
with both the hyperbolic and the de Sitter space
as well as across the boundaries,
an obvious drawback of this extension is that it is a \emph{mixture} 
of Riemannian geometry (the hyperbolic parts) and Lorentzian geometry (the de Sitter part).
Actually, the de Sitter part cannot be directly taken out to 
leave the remaining two hyperbolic parts to form a consistent geometry across the boundary,
so this model does not serve the purpose if one wants the extension to contain only hyperbolic parts.

The crucial difference of our construction of $\mathbb{DH}^n$ is to treat 
the length element $ds$ in the lower sheet $\mathbb{H}^n_{-}$ as being negative,
which makes it possible to use complex analysis
to obtain a well defined volume on an extension that contains only hyperbolic parts.
To our knowledge, this construction we made is the first of its kind in hyperbolic geometry.

\section{Preliminaries}
\label{section_preliminaries}

We recall some basic properties of $\mathbb{H}^n$ and introduce new notions for $\mathbb{DH}^n$.
Let $\mathbb{R}^{n,1}$ be the $(n+1)$-dimensional Minkowski space endowed with a bilinear product
\begin{equation}
\label{equation_bilinear_product}
x\cdot y=-x_0y_0+x_1y_1+\cdots+x_ny_n,
\end{equation}
then in the \emph{hyperboloid model} the hyperbolic space $\mathbb{H}^n$ is defined by
\[ \bigl\{x \in \mathbb{R}^{n,1}: x\cdot x = -1, \quad x_0 > 0 \bigr\},
\]
the upper sheet of a two-sheeted hyperboloid in $\mathbb{R}^{n,1}$.
The bilinear product induces a metric $ds^2=-dx_0^2+dx_1^2+\cdots+dx_n^2$,
and by convention the length element $ds$ is $(ds^2)^{1/2}$ in $\mathbb{R}^{n,1}$.
But in this paper $ds$ can also be $-(ds^2)^{1/2}$ in some other cases
(which is one of the most important features of the construction),
so sometimes we explicitly specify the length element $ds$ for $ds^2$.
Let $\mathbb{R}^{n,1}_{-}$ be a copy of $\mathbb{R}^{n,1}$,
namely it has the same bilinear product (\ref{equation_bilinear_product}) as in $\mathbb{R}^{n,1}$,
but the $ds$ on $\mathbb{R}^{n,1}_{-}$ is the negative of the $ds$ on $\mathbb{R}^{n,1}$
\begin{equation}
\label{equation_negative_metric}
ds=-(-dx_0^2+dx_1^2+\cdots+dx_n^2)^{1/2}.
\end{equation}
Define $\mathbb{H}^n_{-}$ by
\[ \bigl\{x \in \mathbb{R}^{n,1}_{-}: x\cdot x = -1, \quad x_0 < 0 \bigr\},
\]
the lower sheet of a \emph{different} two-sheeted hyperboloid in $\mathbb{R}^{n,1}_{-}$,
namely, $\mathbb{H}^n_{-}$ is not the lower sheet of the hyperboloid in $\mathbb{R}^{n,1}$.
If we treat $\partial\mathbb{H}^n$ as the end of those \emph{half}-lines that lie on the 
\emph{future light cone} $\{x\in\mathbb{R}^{n,1}: x\cdot x=0, \quad x_0>0\}$  in $\mathbb{R}^{n,1}$,
and $\partial\mathbb{H}^n_{-}$ as the end of those half-lines that lie on the 
\emph{past light cone} $\{x\in\mathbb{R}^{n,1}_{-}: x\cdot x=0, \quad x_0<0\}$  in $\mathbb{R}^{n,1}_{-}$,
then by identifying $\partial\mathbb{H}^n$ with $\partial\mathbb{H}^n_{-}$ projectively, 
we glue $\mathbb{H}^n$ and $\mathbb{H}^n_{-}$ together and obtain $\mathbb{DH}^n$.

Notice that the length element $ds$ on $\mathbb{H}^n_{-}$ is negative,
hence $\mathbb{H}^n_{-}$ is \emph{not} isometric to $\mathbb{H}^n$.
However, this change of metric does not affect 
the constant curvature $\kappa$ of $\mathbb{H}^n_{-}$, which is still $-1$.%
\footnote{It should not be confused with changing the metric from $ds^2$ to $-ds^2$,
which does change the sign of the constant curvature $\kappa$.
In this paper, to clear up any confusion,
for a sign change of metric
we always specify whether we are referring to $-ds$ or $-ds^2$.
} 
To have a well defined ``distance'' (which can be complex valued)
between a pair of points in $\mathbb{H}^n$ and $\mathbb{H}^n_{-}$ respectively,
we will show that it is necessary to require the length element $ds$ in $\mathbb{H}^n_{-}$
to be negative as defined in (\ref{equation_negative_metric}).%
\footnote{In fact, this is expected even without using complex analysis.
Let $x$ and $-x$ be the vertices of a half-space in $\mathbb{DH}^1$,
and $y$ and $-y$ be the vertices of a ``smaller'' half-space sitting completely inside the first one.
If we expect all half-spaces to have a fixed length $c$,
then the distance from $-x$ to $-y$ inside $\mathbb{H}^1_{-}$ has to be negative.
}
The associated volume element of $\mathbb{H}^n_{-}$ needs to multiply $(-1)^n$
compared to the $\mathbb{H}^n$ case, so it is also negative if $n$ is odd.

Recall that a half-space in $\mathbb{DH}^n$ is obtained by gluing a half-space in $\mathbb{H}^n$
and its antipodal image in $\mathbb{H}^n_{-}$ along their common boundary.
It can be expressed as
\begin{equation}
\label{equation_half_space_hyperboloid}
\bigl\{x \in \mathbb{DH}^n: \ell(x)x\cdot e \geq 0 \bigr\},
\end{equation}
where $e$ with $e\cdot e=1$ is the inward unit normal to the half-space along the face in the \emph{upper} sheet, 
$x\cdot e$ is the same bilinear product (\ref{equation_bilinear_product}) for both $x_0>0$ and $x_0<0$,
and $\ell(x)=1$ if $x_0>0$ and $\ell(x)=-1$ if $x_0<0$.
The sign of $\ell(x)$ has more to do with the linear space
($\mathbb{R}^{n,1}$ or $\mathbb{R}^{n,1}_{-}$) each of the two sheets is embedded in
than with the sheet itself, 
but for now this definition of $\ell(x)$ is convenient and enough for our purpose.
As $\partial\mathbb{H}^n$ and $\partial\mathbb{H}^n_{-}$ are identified projectively,
it suggests that it is crucial to use $\ell(x)$ to define the half-space in $\mathbb{DH}^n$.
We also call a 1-dimensional half-space in $\mathbb{DH}^1$ a \emph{half circle}.

Let $M^n$ be the spherical, Euclidean, or hyperbolic space 
(including $\mathbb{H}^n_{-}$ as well) of constant curvature $\kappa$.
Unlike in $M^n$ where a half-space is always on the same side of a plane,
in $\mathbb{DH}^n$ notice that in (\ref{equation_half_space_hyperboloid}) 
(see Figure~\ref{figure_half_space}) the half-space's upper and lower parts 
\emph{appear} to be on different sides of a plane $\{x \in \mathbb{R}^{n,1}: x\cdot e = 0 \}$.
Although seemingly counterintuitive, this rather strange property can be explained 
by the fact that the upper and the lower sheets are embedded in different linear spaces
$\mathbb{R}^{n,1}$ and $\mathbb{R}^{n,1}_{-}$ respectively.
However for two half-spaces of $\mathbb{DH}^n$ whose faces intersect, 
the intersecting angle is the same
in both the upper and the lower sheets, which is important for the development of the theory.

By an $n$-dimensional \emph{convex polytope} in $M^n$,
we mean a compact subset which 
can be expressed as a finite intersection of closed half-spaces.
If we remove the compactness restriction and the resulting intersection 
is unbounded, then the intersection is an \emph{unbounded polytope}.
(In $\mathbb{H}^n$ for $n\ge 2$, 
if we relax the convex polytope to also include ideal vertices,
then it still has finite volume.
See Haagerup and Munkholm \cite{HaagerupMunkholm}, 
see also Luo \cite{Luo:continuity} and Rivin \cite{Rivin:volumes}.)
For an overview of polytopes see Ziegler \cite{Ziegler:polytopes}.

\begin{remark}
\label{remark_polytope}
While a convex polytope in $M^n$ is always \emph{convex}
and homeomorphic to a closed ball, this is not so for a polytope in $\mathbb{DH}^n$,
e.g., it can possibly contain a pair of connected components in $\mathbb{H}^n$ 
and $\mathbb{H}^n_{-}$ respectively. 
Also $\mathbb{DH}^n$ itself is a polytope, but without boundary.
\end{remark}

\section{Models for $\mathbb{DH}^n$}
\label{section_models}

To analyze $\mathbb{DH}^n$,
we will apply the same models of $\mathbb{H}^n$ to $\mathbb{DH}^n$
with some simple adjustment,
and we will need those models before we can introduce 
the definition of $V_n(P)$ in Section~\ref{section_definition_volume_P}.
As the details of the models of $\mathbb{H}^n$ are well known
(see, e.g., \cite[Section 7]{Cannon:hyperbolic})
and our adjustment is straightforward, 
we will just jump to the details of the
adjusted models for $\mathbb{DH}^n$ directly.
In fact, by restricting the adjusted models to $\mathbb{H}^n$ only,
one can obtain the original models of $\mathbb{H}^n$.
By convention, we will use the same model names of $\mathbb{H}^n$
to describe $\mathbb{DH}^n$ as well.
But the models for $\mathbb{DH}^n$
are not restricted only to the regions as the model names may suggest,
e.g., the hemisphere model for $\mathbb{DH}^n$
is no longer only restricted to the region of a \emph{hemisphere}
but instead uses the full sphere.

\subsection{Hemisphere model}
\label{section_hemisphere}

We start with the hemisphere model of $\mathbb{H}^n$,
which comes from a \emph{stereographic projection} 
of the upper hyperboloid $\mathbb{H}^n$ from $(-1,0,\dots,0)$ 
to the upper half of the unit $n$-sphere:
$x_0^2+\cdots+x_n^2=1, x_0>0$
with a map 
\begin{equation}
\label{equation_hemisphere_coordinate}
(x_0,x_1,\dots,x_n)\mapsto(1/x_0,x_1/x_0,\dots,x_n/x_0).
\end{equation}

The following argument is crucial for this paper.
The \emph{same} stereographic projection (\ref{equation_hemisphere_coordinate})
also maps the lower hyperboloid $\mathbb{H}^n_{-}$ to the lower half
of the unit sphere with $x_0<0$,
and maps the boundary at infinity projectively
to the unit $(n-1)$-sphere on $x_0=0$ with a natural identification 
$\partial\mathbb{H}^n=\partial\mathbb{H}^n_{-}$.
So $\mathbb{DH}^n$ is mapped one-to-one to the \emph{full} unit sphere. 
The projection  (\ref{equation_hemisphere_coordinate})
maps any pair of antipodal points $x$ and $-x$
to a pair of mirror points $A$ and $A'$ with respect to the plane $x_0=0$
(see Figure~\ref{figure_hemisphere}),
and also maps any $\mathbb{DH}^1$ embedded in $\mathbb{DH}^n$
to a small circle on the $n$-sphere that is perpendicular to $x_0=0$.
A simple but crucial fact is that, as a consequence of the way
a half-space in $\mathbb{DH}^n$ is defined in the hyperboloid model
in (\ref{equation_half_space_hyperboloid}), by (\ref{equation_hemisphere_coordinate})
any half-space in the hemisphere model is an intersection of 
the unit $n$-sphere and a half-space in $\mathbb{R}^{n,1}$
whose face is a vertical plane to $x_0=0$.
In fact, the half-space in the hemisphere model also gives a new explanation that
why in the hyperboloid model a half-space's upper and lower portion
need to appear on different sides of a plane. 
Notice also that a half-space in $\mathbb{DH}^n$ can be fully contained in another half-space,
a feature different from the standard sphere.

\begin{figure}[h]
\centering
  \includegraphics[width=0.6\textwidth]{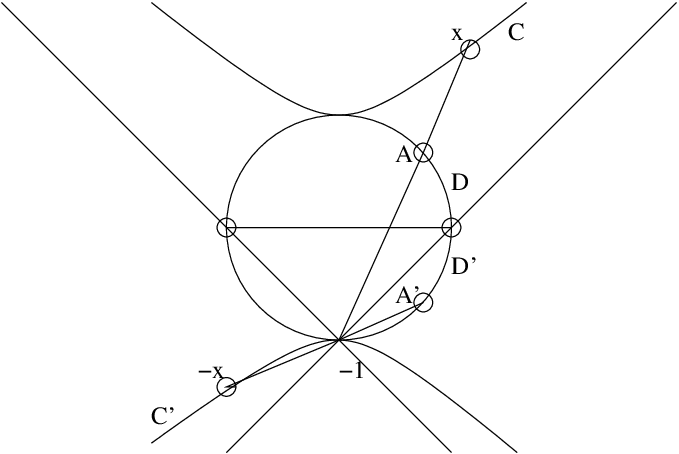}
\caption{The stereographic projection from $-1$ maps $C$ and $C'$ (a half-space in hyperboloid model)
to $D$ and $D'$ (a half-space in hemisphere model) respectively}
\label{figure_hemisphere}
\end{figure}

The associated metric on the upper hemisphere is
\begin{equation}
\label{equation_ds_hemisphere}
ds=(dx_0^2+\cdots+dx_n^2)^{1/2}/x_0,
\end{equation}
and as the length element $ds$ on $\mathbb{H}^n_{-}$ is negative,
the same form applies to the lower hemisphere with $x_0<0$ as well.
In the Euclidean space, let $S^n_r$ be the sphere centered at $O$ with radius $r$,
and $dA_r$ be the \emph{Euclidean} volume element (or ``area element'') on $S^n_r$.
Then at any point on $S^n_r$ with $x_0\ne 0$,
 $dA_r$ is the volume element of the tangent space 
 and can be written as $\frac{dx_1\cdots dx_n}{|\cos\theta|}$,
 where $\theta$ is the angle between the outward unit normal at this point and the $x_0$-axis
 (see Figure~\ref{figure_area_element}).
 As $\cos\theta=x_0/r$, then
\begin{equation}
\label{equation_Euclidean_volume_element}
dA_r=\pm \frac{rdx_1\cdots dx_n}{x_0},
\end{equation}
with the plus sign for $x_0>0$ and the minus sign for $x_0<0$.

\begin{figure}[h]
\centering
\resizebox{.4\textwidth}{!}
  {\input{fig_area_element.pspdftex}}
\caption{Euclidean volume element $dA_r$}
\label{figure_area_element}
\end{figure}

So with $r=1$, the associated volume element of (\ref{equation_ds_hemisphere}),
multiplying $\pm \frac{dx_1\cdots dx_n}{x_0}$ by $\frac{1}{x_0^n}$, is
\begin{equation}
\label{equation_volume_element_hemisphere}
\pm dx_1\cdots dx_n/x_0^{n+1}.
\end{equation}
Notice that for $n$ odd, the coefficient is negative for $x_0<0$.

\subsection{Upper half-space model}
\label{section_upper_half_space}

The upper half-space model for $\mathbb{DH}^n$ comes from a stereographic projection
of the unit $n$-sphere of the hemisphere model from $(0,\dots,0,-1)$ to the plane $x_n=1$
(see Figure~\ref{figure_projection}):
\begin{equation}
\label{equation_upper_half_space_coordinate}
(x_0,\dots,x_{n-1},x_n)\mapsto(2x_0/(x_n+1),\dots,2x_{n-1}/(x_n+1)).
\end{equation}

\begin{figure}[h]
\centering
\resizebox{.4\textwidth}{!}
  {\input{fig_projection.pspdftex}}
\caption{Stereographic projection from hemisphere model to upper half-space model}
\label{figure_projection}
\end{figure}

A half-space is either the inside or the outside of a ball whose center is on $x_0=0$, 
or a Euclidean half-space whose face is a vertical plane to $x_0=0$.
We remark that the upper or lower half-space is not a \emph{half-space} in $\mathbb{DH}^n$.
The geodesics are either straight lines or circles that are perpendicular to $x_0=0$.
For both the upper half-space of $x_0>0$ and the lower half-space of $x_0<0$,
the associated metric is
\begin{equation}
\label{equation_ds_half_space}
ds=(dx_0^2+\cdots+dx_{n-1}^2)^{1/2}/x_0,
\end{equation}
and the associated volume element is 
\begin{equation}
\label{equation_volume_element_space}
dx_0\cdots dx_{n-1}/x_0^n.
\end{equation}
Notice that for $n$ odd the coefficient is negative for $x_0<0$.

\begin{remark}
\label{remark_conformally_equivalent}
We can also view the metric $ds$ (\ref{equation_ds_half_space}) as conformally equivalent to the standard metric 
$ds=(dx_0^2+\cdots+dx_{n-1}^2)^{1/2}$
on the Euclidean space (except at $x_0=0$) by a factor $1/x_0$
(referred as the \emph{conformal factor}). 
While $1/x_0$ is not continuous at $x_0=0$,
when extending to the lower half-space as a complex variable, it preserves the \emph{analyticity} of $x_0$
and thus in some sense makes the metric in (\ref{equation_ds_half_space}) ``conformal'' across $x_0=0$.
This also makes it possible to use complex analysis to compute the integral 
across $\partial\mathbb{H}^n$ with respect to $x_0$,
so heuristically making it a better choice for the extension than other factors like $|1/x_0|$.
It is also worth noting that this argument works not only for the Riemannian metric,
but also for the pseudo-Riemannian metric like the Lorentz metric.
\end{remark}

\begin{remark}
\label{remark_hemisphere_half_space}
Notice the similarity between the metric above of the upper half-space model 
(\ref{equation_ds_half_space}) and 
the metric of the hemisphere model (\ref{equation_ds_hemisphere}).
It suggests that a hemisphere model for $\mathbb{DH}^n$ can also be
viewed as \emph{embedded} in an upper half-space model 
for a higher dimensional $\mathbb{DH}^{n+1}$,
which will be helpful in understanding the proof of the \SDF{} 
for $\mathbb{DH}^n$ in later sections.
\end{remark}

\subsection{Klein model}
\label{section_Klein}

Finally, the \emph{Klein model} for $\mathbb{DH}^n$ comes from 
a central projection from the origin 
that maps the two-sheeted hyperboloid $\mathbb{H}^n$ and $\mathbb{H}^n_{-}$
to the disk: $x_1^2+\cdots+x_n^2<1, x_0=1$.
The boundary at infinity is mapped projectively 
to the boundary of the disk with a natural identification 
$\partial\mathbb{H}^n=\partial\mathbb{H}^n_{-}$.
The projection forms a double covering%
\footnote{In this paper the term ``double covering'' is used loosely 
to only refer to the open part of a space,
usually with the boundary points ignored,
and the metrics on the two covers need not agree.
}
of the disk,
where $\mathbb{H}^n$ is mapped to the upper and $\mathbb{H}^n_{-}$ to the lower cover,
with the length element $ds$ on the lower cover being negative.
We will use this model mainly for \emph{visualization} purposes
as a polytope always has ``flat'' faces as a Euclidean polytope does,
but we will not use its metric to compute the volume
which is rather cumbersome.

\section{A definition of $V_n(P)$ for polytopes in $\mathbb{DH}^n$}
\label{section_definition_volume_P}

\subsection{Using the upper half-space model}
\label{section_upper_half_space_complex}

In the upper half-space model, for a Lebesgue measurable set $U$ in $\mathbb{H}^n$,
the standard hyperbolic volume of $U$ is defined by integrating 
the volume element (\ref{equation_volume_element_space}) in $\mathbb{R}^n$
\[V_n(U)=\int_{U\subset\mathbb{R}^n}\frac{dx_0\cdots dx_{n-1}}{x_0^n}.
\]
The integral also applies to regions in $\mathbb{H}^n_{-}$.
For a set $U$ in $\mathbb{DH}^n$, if both $U_{+}$ in $\mathbb{H}^n$ 
and $U_{-}$ in $\mathbb{H}^n_{-}$ has \emph{finite} standard hyperbolic volume,
then we say that $U$ has finite standard hyperbolic volume
and define $V_n(U)$ by $V_n(U_{+})+V_n(U_{-})$,
and let $\mathcal{U}_0$ be the collection of all those sets. 
For $U\in\mathcal{U}_0$, $V_n(U)$ is invariant under isometry
(recall that an isometry of $\mathbb{DH}^n$ is an isometry of $\mathbb{H}^n$
that also preserves the counterpart in $\mathbb{H}^n_{-}$,
which is the mirror image in $\mathbb{H}^n_{-}$ in the upper half-space model).
But for regions sitting across $\partial\mathbb{H}^n$, the integral is not well defined at $x_0=0$.
To fix this issue, we want to define a volume as the integral of a complex perturbation 
of the volume element (as a complex measure on $\mathbb{R}^n$).
But we shall note that in the complex perturbation the underlying space is still real, even for coordinate $x_0$.
So in this paper complex analysis is used in a very limited way mainly to compute the integral of 
complex valued functions in the real space.
See also a similar but slightly different methodology called ``$\epsilon$-approximation''
employed by Cho and Kim \cite{ChoKim} using the Klein model.

For any $\epsilon\ne 0$, let $\mathbb{R}^n$ be endowed with a complex valued ``Riemannian metric''
\begin{equation}
\label{equation_ds_half_space_epsilon}
ds_{\epsilon}=(dx_0^2+\cdots+dx_{n-1}^2)^{1/2}/(x_0-\epsilon i),
\end{equation}
namely, the inner product on the tangent space at a point is
$1/(x_0-\epsilon i)^2$ times the standard Euclidean inner product.
It is conformally equivalent to the standard Euclidean metric
$ds=(dx_0^2+\cdots+dx_{n-1}^2)^{1/2}$ by a factor $1/(x_0-\epsilon i)$,
so the angle at each point is the same as the Euclidean metric, independent of the value $1/(x_0-\epsilon i)$.
The associated volume element of (\ref{equation_ds_half_space_epsilon})
is $\frac{dx_0\cdots dx_{n-1}}{(x_0-\epsilon i)^n}$.
One may view the (singular) metric (\ref{equation_ds_half_space})
on $\mathbb{R}^n$ as a limit of $ds_{\epsilon}$ with $\epsilon\to 0^+$,
and we study the volume as a limit of the complex valued volume induced by $ds_{\epsilon}$.

In the following we denote $U$ a Lebesgue measurable set in $\mathbb{R}^n$,
and define
\begin{equation}
\label{equation_volume_P_half_space_epsilon}
\mu_{\epsilon}(U)= \int_{U\subset \mathbb{R}^n}\frac{dx_0\cdots dx_{n-1}}{(x_0-\epsilon i)^n},
\quad\mu(U)=\lim_{\epsilon\to 0^+} \mu_{\epsilon}(U),
\end{equation}
whenever the integral exists (finite).

We call $U$ $\mu$-\emph{measurable} if $\mu(U)$ exists,
and let $\mathcal{U}$ be the collection of $\mu$-measurable sets.
Clearly $\mu$ is finitely additive on $\mathcal{U}$.
If $U$ is in a finite region in $\mathbb{R}^n$, 
as $|\frac{1}{(x_0-\epsilon i)^n}| \le |\frac{1}{\epsilon^n}|$,
then it automatically guarantees the existence of $\mu_{\epsilon}(U)$, 
but not of $\mu(U)$. 
We remark that $\mathbb{R}^n$ is not  $\mu$-measurable,
as $\mu_{\epsilon}(\mathbb{R}^n)$ does not exist.

For $U\in\mathcal{U}_0$ (namely, both $U_{+}$ and $U_{-}$ has finite standard hyperbolic volume,
and $V_n(U)=V_n(U_{+})+V_n(U_{-})$),
as $|\frac{1}{(x_0-\epsilon i)^n}| \le |\frac{1}{x_0^n}|$,
the Lebesgue dominated convergence theorem applies,
thus $\mu(U)$ exists and $\mu(U)=V_n(U)$.
So $\mathcal{U}_0\subset\mathcal{U}$.
Our goal is to extend $V_n(U)$ to more sets beyond $\mathcal{U}_0$,
preferably by using $\mu(U)$ on some $\mu$-measurable sets $U$,
but some issues need to be addressed.
First, any extension of $V_n(U)$ should be invariant under isometry, but we don't know if $\mu(U)$ 
is invariant under isometry for all $\mu$-measurable sets in $\mathcal{U}$.
Even if so, $\mathcal{U}$ is not an algebra in the sense that 
there are $\mu$-measurable sets $U$ and $U'$ such that
$U\cap U'$ is not $\mu$-measurable.%
\footnote{
Let $U$ be a unit ball centered at the origin.
For $U'$, let $U'_{+}=U_{+}$ and $U'_{-}$ be a translation of $U_{-}$ 
such that $U'_{-}$ is still centered on $x_0=0$ but $U'_{-}\cap U_{-}=\varnothing$. 
By Proposition~\ref{proposition_mu_u_P}, $U$ is $\mu$-measurable;
by (\ref{equation_volume_P_half_space_epsilon}) we have
$\mu_{\epsilon}(U')=\mu_{\epsilon}(U)$, so $U'$ is also $\mu$-measurable.
But  $U\cap U'=U_{+}$ is a half-space in $\mathbb{H}^n$, not $\mu$-measurable.
}

We will show that if we only extend $V_n(U)$ to those sets generated by half-spaces
in $\mathbb{DH}^n$ and elements in $\mathcal{U}_0$, then all those issues above will be resolved.
Without affecting the computation of volume,
here a half-space can be either a closed or open half-space,
and let $\mathcal{H}$ be the algebra over $\mathbb{DH}^n$
generated by half-spaces in $\mathbb{DH}^n$.
Recall from above that $\mathbb{R}^n$ is not $\mu$-measurable,
and as $\mathbb{DH}^n$ only has an extra point $\infty$ compared to $\mathbb{R}^n$,
so it does not affect the computation of the $\mu_{\epsilon}$-measure and
$\mu$-measure of $\mathbb{DH}^n$, and thus $\mathbb{DH}^n \not\in \mathcal{U}$.
In order to have a finite value of $V_n(\mathbb{DH}^n)$ as in Theorem~\ref{theorem_total_volume},
the definition of $V_n(U)$ need to deviate from $\mu(U)$ for certain sets $U$
when $\mu(U)$ does not exist.
However, this is a minor issue that will be handled in Definition~\ref{definition_volume_P}.

\begin{proposition}
\label{proposition_mu_u_P}
For $n\ge 1$ and $P\in\mathcal{H}$ in $\mathbb{DH}^n$,
if $P$ is in a finite region in $\mathbb{R}^n$, then $\mu(P)$ exists and 
is invariant under isometry.
\end{proposition}

Proposition~\ref{proposition_mu_u_P} is a special case of Theorem~\ref{theorem_volume_invariant}
that will be proved later. Now we define $V_n(P)$ as  a measurable algebra on $\mathcal{H}$. 

\begin{definition}
\label{definition_volume_P}
For $n\ge 1$ and $P\in\mathcal{H}$ in $\mathbb{DH}^n$,
if $P$ is in a finite region in $\mathbb{R}^n$, define $V_n(P)$ by $\mu(P)$;
otherwise we first map $P$ to a finite region in $\mathbb{R}^n$ by an isometry $g$ of $\mathbb{DH}^n$,
then define $V_n(P)$ by $\mu(g(P))$.
For $n=0$, define $V_0(\mathbb{DH}^0)=2$, the number of points in $\mathbb{DH}^0$.
\end{definition}

For $\mathbb{DH}^n$ whose volume cannot be computed as a single piece by Definition~\ref{definition_volume_P},
we can cut $\mathbb{DH}^n$ into two half-spaces
and compute the volume of each piece using Definition~\ref{definition_volume_P}
and then sum them up.
Proposition~\ref{proposition_mu_u_P} ensures that $V_n(P)$ is well defined
and finitely additive on $\mathcal{H}$.
(We will show in Example~\ref{example_not_countably_additive_P} that $V_n(P)$ 
is not countably additive.)

We check some simple cases.
For $n=1$, if $P$ is a half circle in $\mathbb{DH}^1$ and satisfies $-a\le x_0\le a$, we have
$\mu_{\epsilon}(P)=\int_{-a}^a\frac{dx_0}{x_0-\epsilon i}=\int_{-a-\epsilon i}^{a-\epsilon i}\frac{dx_0}{x_0}$.
For $\epsilon>0$, the path is \emph{below} the origin in the complex plane of $x_0$, and thus
$V_1(P)=\mu(P)=\lim_{\epsilon\to 0^+}\mu_{\epsilon}(P)=\pi i$,
independent of the value of $a$.
(We remark that this is not the unique way to define $V_1(P)$,
say if we define $V_1(P)$ by $\lim_{\epsilon\to 0^-}\mu_{\epsilon}(P)$ instead,
then $V_1(P)=-\pi i$.)
As $\mathbb{DH}^1$ is made of two half circles, this gives
\begin{equation}
\label{equation_total_volume_dh1}
V_1(\mathbb{DH}^1)=2\pi i=i V_1(\mathbb{S}^1).
\end{equation}
So under this definition
the total length of a 1-dimensional polytope in $\mathbb{DH}^1$
only takes the values of 0, $\pi i$, or $2\pi i$.

In general, if $P\in\mathcal{H}$ is in a finite region in $\mathbb{R}^n$,
and denote by $P_{+}$ and $P_{-}$ the upper and the lower portion of $P$ respectively,
we have (be aware of the signs below)
\begin{equation}
\label{equation_mu_epsilon_P_plus}
\mu_{\epsilon}(P)=\mu_{\epsilon}(P_{+})+\mu_{\epsilon}(P_{-})
=\mu_{\epsilon}(P_{+})+(-1)^n \mu_{-\epsilon}(P_{+}).
\end{equation}
Thus $\mu_{\epsilon}(P)=\mu_{-\epsilon}(P)$ for $n$ even
and $\mu_{\epsilon}(P)=-\mu_{-\epsilon}(P)$ for $n$ odd,
which suggests that the choice of the sign of $\epsilon$ only affects 
the definition of $V_n(P)$ for $n$ odd but not for $n$ even.

If $P_{+}$ has finite standard hyperbolic volume, 
for $n$ even then $V_n(P_{+})=V_n(P_{-})$,
thus $V_n(P)=2V_n(P_{+})$;
for $n$ odd then $V_n(P_{+})=-V_n(P_{-})$,
thus they cancel each other out and therefore $V_n(P)=0$.
But $V_n(P)$ is not trivial on $\mathcal{H}$ for $n$ odd,
as $V_n(P)$ is not zero when $P$ contains a non-trivial portion of $\partial\mathbb{H}^n$.

\subsection{Measure theory on $\mathbb{DH}^n$}

To make $V_n(P)$ on $\mathcal{H}$ truly an \emph{extension} 
of the standard hyperbolic volume, 
we show that the definition of $V_n(U)$ can be further extended to $\mathcal{H}'$,
the algebra over $\mathbb{DH}^n$ 
generated by $\mathcal{H}$ and $\mathcal{U}_0$.
If $U\in\mathcal{H}'$, denote $U^c$ the complement of $U$ in $\mathbb{DH}^n$
(not in $\mathbb{H}^n$, even if $U$ is entirely in $\mathbb{H}^n$).

\begin{proposition}
\label{proposition_mu_u_U}
For $n\ge 1$,
by Definition~\ref{definition_volume_P} of $V_n(P)$ on $\mathcal{H}$ in $\mathbb{DH}^n$,
the definition of $V_n(U)$ can be uniquely extended to $\mathcal{H}'$ as a measurable algebra.
Particularly for $P\in\mathcal{H}$ and $U_0\in\mathcal{U}_0$,
we have $V_n(P\setminus U_0)=V_n(P)-V_n(P\cap U_0)$.
\end{proposition}

\begin{proof}
For any $U\in\mathcal{H}'$, 
assume $U$ is generated by half-spaces $H_1,\dots, H_k$ and $U_1,\dots, U_l\in\mathcal{U}_0$.
Then by a property in set theory,
$U$ is the disjoint union of regions $E$ in the form $(\bigcap_{i=1}^k A_i)\cap (\bigcap_{j=1}^l B_j)$,
where $A_i$ is either $H_i$ or $H_i^c$ and $B_j$ is either $U_j$ or $U_j^c$.
For any of the region $E$, if at least one $B_j$ is $U_j$, 
then $E\in\mathcal{U}_0$ and thus $V_n(E)$ exists.

Now for a region $E$ assume all $B_j$ is $U_j^c$, and denote $\bigcup_{j=1}^l U_j$ by $U_0$
and $\bigcap_{i=1}^k A_i$ by $P$.
Then $P\in\mathcal{H}$ and $U_0\in\mathcal{U}_0$, and
\[E=P\cap (\bigcap U_j^c)=P\cap (\bigcup U_j)^c=P\cap U_0^c=P\setminus U_0.
\]
As $P$ is the disjoint union of $P\cap U_0$ and $E=P\setminus U_0$,
and both $V_n(P)$ (by Definition~\ref{definition_volume_P}) and $V_n(P\cap U_0)$ are well defined,
we define $V_n(E)=V_n(P)-V_n(P\cap U_0)$.
Sum up all those regions $E$, then $V_n(U)$ is uniquely defined on $\mathcal{H}'$ as a measurable algebra,
and is also invariant under isometry.
\end{proof}

\begin{remark}
In general, if $V_n(U)$ is already extended to an algebra $\mathcal{F}$ over $\mathbb{DH}^n$,
then by the same method in the proof of Proposition~\ref{proposition_mu_u_U},
we can show that $V_n(U)$ is also well defined on $\mathcal{F}'$, 
the algebra over $\mathbb{DH}^n$ generated by $\mathcal{F}$ and $\mathcal{U}_0$.
\end{remark}

To summarize, to ensure that $V_n(U)$ is well defined on $\mathcal{H}'$,
by Proposition~\ref{proposition_mu_u_U}
all left to do is to verify Proposition~\ref{proposition_mu_u_P},
which is a special case of Theorem~\ref{theorem_volume_invariant}.

\begin{remark}
\label{remark_union_polytopes}
As $P\in\mathcal{H}$ is generated by half-spaces in $\mathbb{DH}^n$,
so $P$ is the disjoint union of the intersection of half-spaces
(see also the proof of Proposition~\ref{proposition_mu_u_U} for reference),
thus $P$ can be cut into a union of polytopes.
So for convenience from now on for $P\in\mathcal{H}$ we can focus our attention on \emph{polytopes} $P$ only.
\end{remark}

\subsection{Using the hemisphere model}

We need to compute $V_n(P)$ in the hemisphere model as well, which we address next.
The Euclidean volume element on $\mathbb{S}^n$ can be written as $\pm \frac{dx_1\cdots dx_n}{x_0}$,
with the plus sign for $x_0>0$ and the minus sign for $x_0<0$ (see (\ref{equation_Euclidean_volume_element})).
When the above $ds_{\epsilon}$ in (\ref{equation_ds_half_space_epsilon})
 (of a higher dimensional $\mathbb{R}^{n+1}$) is restricted to $\mathbb{S}^n$,
the induced volume element on $\mathbb{S}^n$,
multiplying $\pm \frac{dx_1\cdots dx_n}{x_0}$ by $\frac{1}{(x_0-\epsilon i)^n}$,
is $\pm\frac{dx_1\cdots dx_n}{x_0(x_0-\epsilon i)^n}$.
Let $P$ be a polytope in $\mathbb{DH}^n$ in the hemisphere model, 
for $n\ge 0$ define
\begin{equation}
\label{equation_volume_P_hemisphere_epsilon}
\mu_{h,\epsilon}(P)= \int_{P\subset \mathbb{S}^n}\pm\frac{dx_1\cdots dx_n}{x_0(x_0-\epsilon i)^n},
\quad\mu_h(P)=\lim_{\epsilon\to 0^+} \mu_{h,\epsilon}(P).
\end{equation}

For the general case that a polytope $P$ in $\mathbb{DH}^n$
is on an \emph{non}-unit $n$-sphere $S^n_r$ with radius $r>0$
in $\mathbb{R}^{n+1}$ (and with the center on $x_0=0$),
we also provide a formula for $\mu_{h,\epsilon}(P)$,
which will be particularly useful when the $n$-sphere is treated as a subspace of a higher dimensional space.
Note that the constant curvature $\kappa$ of $S^n_r$ is still $-1$, independent of $r$.
The Euclidean volume element on $S^n_r$ is $\pm \frac{rdx_1\cdots dx_n}{x_0}$ 
(see (\ref{equation_Euclidean_volume_element})).
When $ds_{\epsilon}$ in (\ref{equation_ds_half_space_epsilon})
(of a higher dimensional $\mathbb{R}^{n+1}$) is restricted to $S^n_r$,
the induced volume element on $S^n_r$,
multiplying $\pm \frac{rdx_1\cdots dx_n}{x_0}$ by $\frac{1}{(x_0-\epsilon i)^n}$,
is $\pm\frac{rdx_1\cdots dx_n}{x_0(x_0-\epsilon i)^n}$.
Define
\begin{equation}
\label{equation_volume_P_hemisphere_r_epsilon}
\mu_{h,\epsilon}(P)
=\int_{P\subset S^n_r}\pm\frac{rdx_1\cdots dx_n}{x_0(x_0-\epsilon i)^n},
\quad\mu_h(P)=\lim_{\epsilon\to 0^+} \mu_{h,\epsilon}(P).
\end{equation}

In the hemisphere model,
for technical reasons we also introduce a variant of $\mu_h(P)$
that is more convenient to use in the later proof of the \SDF{} for $\mathbb{DH}^n$.
We directly work with the general case that $P$ is on an $n$-sphere $S^n_r$,
and define
\begin{equation}
\label{equation_volume_P_hemisphere_r_epsilon_variant}
\mu'_{h,\epsilon}(P)= \int_{P\subset S^n_r}\pm\frac{rdx_1\cdots dx_n}{(x_0-\epsilon i)^{n+1}},
\quad\mu'_h(P)=\lim_{\epsilon\to 0^+} \mu'_{h,\epsilon}(P).
\end{equation}
We will later show that $\mu_h(P)=\mu'_h(P)$ 
(Proposition~\ref{proposition_mu_h_prime_invariant}).

For $n=0$, in the upper half-space model, 
the definition of $\mu(\mathbb{DH}^0)$ is not needed for our purpose.
In the hemisphere model, for the 0-sphere $S^0_r$ with radius $r$,
by (\ref{equation_volume_P_hemisphere_r_epsilon})
and (\ref{equation_volume_P_hemisphere_r_epsilon_variant}) we have
\begin{equation}
\label{equation_epsilon_0}
\mu_{h,\epsilon}(S^0_r)=2
\quad\text{and}\quad
\mu'_{h,\epsilon}(S^0_r)=\frac{2r^2}{r^2+\epsilon^2}.
\end{equation}
So  
\begin{equation}
\label{equation_epsilon_0_bound}
\mu'_{h,\epsilon}(S^0_r)\le 2
\quad\text{and}\quad
\mu_h(S^0_r)=\mu'_h(S^0_r)=2.
\end{equation}

Notice that for $n\ge 1$, the values of $\mu_{\epsilon}(P)$, $\mu_{h,\epsilon}(P)$ and $\mu'_{h,\epsilon}(P)$ 
are not invariant under isometry.
But for $\mu(P)$, $\mu_h(P)$ and $\mu'_h(P)$, we have the following important property.

\begin{theorem}
\label{theorem_volume_invariant}
Let $P$ be a polytope in $\mathbb{DH}^n$.
In the upper half-space model we assume $P$ is in a finite region in $\mathbb{R}^n$.
For $n=0$, we are only concerned with the hemisphere model on $S^0_r$
but not the upper half-space model.
\begin{enumerate}
\item[(1)] Then $\mu(P)$, $\mu_h(P)$ and $\mu'_h(P)$ exist and are invariant under isometry,
including isometries between the models.
\item[(2)] \emph{(Uniform boundedness for a fixed $m$)}
For a fixed $m$,
if $P$ is the intersection of at most $m$ half-spaces, 
then $\mu_{\epsilon}(P)$, $\mu_{h,\epsilon}(P)$ and $\mu'_{h,\epsilon}(P)$
are uniformly bounded for all $P$ and $\epsilon>0$.
\end{enumerate}
\end{theorem}

Proposition~\ref{proposition_mu_u_P} is a special case of Theorem~\ref{theorem_volume_invariant}.
See Remark~\ref{remark_union_polytopes}.

\begin{remark}
\label{remark_hemisphere_model_r} 
In the hemisphere model, Theorem~\ref{theorem_volume_invariant}
works \emph{uniformly} for all $n$-spheres $S^n_r$ with $r>0$.
This is because if we switch from $S^n_r$ to the unit $n$-sphere $\mathbb{S}^n$ by change of variables,
$r$ and $\epsilon$ can be combined into a single variable $\epsilon/r$,
so the information of $r$ is ``absorbed'' by $\epsilon$.
Then for convenience we can switch to work on $\mathbb{S}^n$ with $\epsilon$ only but not $r$.
\end{remark}

As the proof of Theorem~\ref{theorem_volume_invariant} is intertwined with 
our proof of the \SDF{} for $\mathbb{DH}^n$ (Theorem~\ref{theorem_schlafli}) in an inductive manner,
it is placed much later in Section~\ref{section_volume_invariant_proof}.
The induction step runs through 
Section~\ref{section_schlafli_special}--\ref{section_invariance_properties}.

\section{Proof overview of Theorem~\ref{theorem_volume_invariant} and \ref{theorem_schlafli}}

The \SDF{} plays a crucial role in this paper.
Consider a family of $n$-dimensional polytopes $P$ 
which vary smoothly in $M^n$ of constant curvature $\kappa$.
For each $(n-2)$-dimensional face $F$, let $\theta_F$ be the dihedral angle at $F$.
Then for $n\geq 2$, the \SDF{} states that
\begin{equation}
\label{equation_schlafli}
\kappa\cdot dV_n(P)=\frac{1}{n-1}\sum_{F}V_{n-2}(F)\,d\theta_F,
\end{equation}
where the sum is taken over all $(n-2)$-faces $F$ of $P$.
For $n-2=0$, $V_0(F)$ is the number of points in $F$.
Though the original \SDF{} did not include $\mathbb{H}^n_{-}$,
it is clear that the \SDF{} is true for $\mathbb{H}^n_{-}$ with $\kappa=-1$.

The \SDF{} was also generalized into other forms. 
Su{\'a}rez-Peir{\'o}~\cite{Suarez:deSitter} proved a \SDF{} for simplices in the de Sitter space
with Riemannian faces. 
Rivin and Schlenker~\cite{RivinSchlenker} obtained a smooth analogue of the \SDF{}
for the volume bounded by a hypersurface moving in a general Einstein manifold.
A \SDF{} for simplices in $M^n$ based on edge lengths
was obtained by the author \cite[Proposition 2.11]{Zhang:rigidity}.
For $\mathbb{DH}^n$, it has constant curvature $\kappa=-1$ everywhere
except on $\partial\mathbb{H}^n=\partial\mathbb{H}^n_{-}$,
so with a proper setting
it seems natural to expect a \SDF{} in some form for $\mathbb{DH}^n$ as well.
In Theorem~\ref{theorem_schlafli} we show that (\ref{equation_schlafli}) 
also holds for $\mathbb{DH}^n$.

Milnor \cite{Milnor:Schlafli} gave a very transparent proof
of the \SDF{} for $M^n$, and particularly for the hyperbolic case,
we essentially adopt its methodology with some adjustments,
and make the argument work for $\mathbb{DH}^n$ as well.
Our proof emphasizes on dealing with the differences between
$\mathbb{DH}^n$ and $\mathbb{H}^n$.
But unlike in $\mathbb{H}^n$ where convergence issues are generally not the main concern of proving \SDF{},
in $\mathbb{DH}^n$ we need to make efforts to show that $V_n(P)$ is well defined
(Theorem~\ref{theorem_volume_invariant}),
which makes our proof much longer than Milnor's original proof.

The main idea of our proof of Theorem~\ref{theorem_volume_invariant} is to compute
the $n$-dimensional volume of a polytope $P$ in $\mathbb{DH}^n$ by integrating the volumes
of lower dimensional faces through various special versions of \SDF{} for $\mathbb{DH}^n$.
This is the reason that why we prove both Theorem~\ref{theorem_volume_invariant} 
and \ref{theorem_schlafli} together with the proofs intertwined (with the focus more on the former).
Our proof runs by induction on $n$,
with dimension 0 and 1 of Theorem~\ref{theorem_volume_invariant} 
verified in Lemma~\ref{lemma_volume_invariant_0_1}.
The induction step runs through the entire
Section~\ref{section_schlafli_special} and \ref{section_invariance_properties}
by proving the statements of $\mu(P)$, $\mu_h(P)$ and $\mu'_h(P)$ inductively.
Those intermediate results that rely on the induction assumption 
are stated with ``Assuming Theorem~\ref{theorem_volume_invariant} is true up to \dots'' at the beginning.
But for those results not dependent on the induction assumption,
they can be freely applied to higher dimensions.
The proof of Theorem~\ref{theorem_schlafli} follows next
in Section~\ref{section_schlafli_double_proof}.

In the upper half-space model, to compute $\mu(P)$,
for convenience we will assume $P$ is in a finite region in $\mathbb{R}^n$ unless specified.

\section{Theorem~\ref{theorem_volume_invariant} for dimension 0 and 1}

We first prove Theorem~\ref{theorem_volume_invariant} for dimension 0 and 1.
The proof below is mainly computational and not used in other sections, 
the reader not interested in technicalities may skip it for now and come back later to check the details.

\begin{lemma}
\label{lemma_volume_invariant_0_1}
Theorem~\ref{theorem_volume_invariant} is true for dimension 0 and 1.
\end{lemma}

\begin{proof}
For $n=0$, by (\ref{equation_epsilon_0}) and (\ref{equation_epsilon_0_bound}),
Theorem~\ref{theorem_volume_invariant} is true.
For $n=1$, as a polytope in $\mathbb{DH}^1$ is either a half circle,
or the union or difference of two half circles, 
so we only need to show that for a half circle $P$, 
it satisfies the invariance and boundedness properties.

In the upper half-space model, assume $-a\le x_0\le a$ for $P$, then
\[\mu_{\epsilon}(P)=\int_{-a}^a\frac{dx_0}{x_0-\epsilon i}
=\int_0^a\left(\frac{1}{x_0-\epsilon i}-\frac{1}{x_0+\epsilon i}\right)dx_0
=\int_0^a\frac{2\epsilon i}{x_0^2+\epsilon^2}dx_0.
\]
Let 
\begin{equation}
\label{equation_b_x_t}
b(x,t):=\frac{2t}{x^2+t^2}.
\end{equation} 
As $\int_0^{\infty} b(x,1) dx=\pi$,
by change of variable (substitute $x_0$ by $x\epsilon$), then
\[ |\mu_{\epsilon}(P)|=\int_0^{a/\epsilon} \frac{2}{x^2+1} dx
=\int_0^{a/\epsilon} b(x,1) dx \le \pi
\]
and 
\[\mu(P)=\lim_{\epsilon\to 0^+} \mu_{\epsilon}(P)
=i \lim_{\epsilon\to 0^+} \int_0^{a/\epsilon} b(x,1) dx
=\pi i.
\]
So $\mu_{\epsilon}(P)$ is uniformly bounded and $\mu(P)=\pi i$.

In the hemisphere model, without loss of generality we use the unit 1-sphere $\mathbb{S}^1$
(see Remark~\ref{remark_hemisphere_model_r}).
Notice that any polytope in $\mathbb{DH}^1$ can be cut into two parts with $x_1\le 0$ and $x_1\ge 0$,
and each part is either a half circle or the difference of two half circles.
By symmetry, without loss of generality we further assume $P$ is a half circle
that satisfies $0\le t_0\le x_1\le 1$ for some $t_0$.
Let $g$ be a stereographic projection of the unit circle from $(0,-1)$ to the space $x_1=1$
(we use coordinate $y$ for this space, see Figure~\ref{figure_projection}): $g(x_0,x_1)=y$.
By  (\ref{equation_upper_half_space_coordinate}) $y=\frac{2x_0}{x_1+1}$.

We first consider $\mu_h(P)$ and $\mu_{h,\epsilon}(P)$ 
(then next $\mu'_h(P)$ and $\mu'_{h,\epsilon}(P)$).
By  (\ref{equation_volume_P_hemisphere_epsilon})
\[\mu_{h,\epsilon}(P)
=\int_{P\subset \mathbb{S}^1}\pm\frac{dx_1}{x_0(x_0-\epsilon i)}  
=\int_{P\subset \mathbb{S}^1}\pm\frac{x_0}{x_0-\epsilon i}\frac{dx_1}{x_0^2},
\]
with the plus sign for $x_0>0$ and the minus sign for $x_0<0$.
As $g$ is an isometry, so $(g^{-1})^{\ast}$ maps the volume element $\pm\frac{dx_1}{x_0^2}$ 
in the hemisphere model (on $\mathbb{S}^1$) into the volume element $\frac{dy}{y}$ 
in the upper half-space model with the appropriate orientation of each coordinate system.
Thus
\begin{equation}
\label{equation_mu_h_half_circle}
\mu_{h,\epsilon}(P)
=\int_{g(P)\subset \mathbb{R}}\frac{x_0}{x_0-\epsilon i}\frac{dy}{y}
=\int_{g(P)\subset \mathbb{R}}\frac{dy}{y-\frac{y}{x_0}\epsilon i}
=\int_{g(P)\subset \mathbb{R}}\frac{dy}{y-c(y)\epsilon i},
\end{equation}
where $c(y):=\frac{y}{x_0}$.
By (\ref{equation_upper_half_space_coordinate}) $c(y)=\frac{2}{x_1+1}$,
particularly $c(0)=1$.
As $0\le t_0\le x_1\le 1$, therefore $1\le c(y) \le \frac{2}{t_0+1}\le 2$.
Denote $\frac{2}{t_0+1}$ by $c_0$, so $1\le c(y)\le c_0\le 2$.

Assume $-a\le y\le a$ for $g(P)$, 
as $c(y)=c(-y)$, then by (\ref{equation_mu_h_half_circle})
\begin{equation}
\label{equation_mu_h_epsilon_P}
\mu_{h,\epsilon}(P)
=\int_0^a\left( \frac{1}{y-c(y)\epsilon i} - \frac{1}{y+c(y)\epsilon i} \right) dy
=\int_0^a\frac{2c(y)\epsilon i}{y^2+c(y)^2\epsilon^2} dy.
\end{equation}
By (\ref{equation_b_x_t}), for $\epsilon>0$
\begin{equation}
\label{equation_mu_h_epsilon_estimate}
|\mu_{h,\epsilon}(P)| = \int_0^a b(y,c(y)\epsilon) dy.
\end{equation}
Notice that for a fixed $y>0$, for $t>0$ the function $b(y,t)=\frac{2t}{y^2+t^2}$ peaks at $t=y$
and thus is increasing if $0< t\le y$ and decreasing if $t\ge y$.

Now fix $\epsilon>0$.
To show that (\ref{equation_mu_h_epsilon_estimate}) is uniformly bounded,
without loss of generality we assume $a>c_0\epsilon$.
As $1\le c(y) \le c_0$ (as shown above) and thus 
 $\epsilon \le c(y)\epsilon \le c_0\epsilon$,
accordingly we break the right side of (\ref{equation_mu_h_epsilon_estimate}) into three parts:
$y\le \epsilon$, $y\ge c_0\epsilon$, and $\epsilon< y< c_0\epsilon$.
For a fixed $y\le \epsilon$, because $b(y,t)$ is decreasing for $t$ if $t\ge y$ and $c(y)\epsilon\ge\epsilon$,
so $b(y,c(y)\epsilon)\le b(y,\epsilon)$, therefore
\begin{equation}
\label{equation_integration_b_eqsilon_1}
\int_0^{\epsilon} b(y,c(y)\epsilon) dy
\le \int_0^{\epsilon} b(y,\epsilon) dy
= \int_0^1 b(x,1) dx
\le \pi,
\end{equation}
where the second step is by change of variable.

Similarly for a fixed $y\ge c_0\epsilon$, 
because $b(y,t)$ is increasing for $t$ if $0<t\le y$ and $c(y)\epsilon\le c_0\epsilon$,
so $b(y,c(y)\epsilon)\le b(y,c_0\epsilon)$, hence
\begin{equation}
\label{equation_integration_b_eqsilon_2}
\int_{c_0\epsilon}^a b(y,c(y)\epsilon) dy
\le \int_{c_0\epsilon}^a b(y,c_0\epsilon) dy
= \int_1^{a/c_0\epsilon} b(x,1) dx
\le \pi,
\end{equation}
where the second step is again by change of variable.

Finally, for $\epsilon<y<c_0\epsilon$, because $1\le c(y)$ (as shown above),
\begin{equation}
\label{equation_integration_b_eqsilon_3}
\int_{\epsilon}^{c_0\epsilon} b(y,c(y)\epsilon) dy
\le \int_{\epsilon}^{c_0\epsilon} \frac{2c(y)\epsilon}{c(y)^2\epsilon^2} dy
\le \int_{\epsilon}^{c_0\epsilon} \frac{2\epsilon}{\epsilon^2} dy
=2(c_0-1).
\end{equation}
Adding (\ref{equation_integration_b_eqsilon_1}), (\ref{equation_integration_b_eqsilon_2}),
and (\ref{equation_integration_b_eqsilon_3}) together,
by (\ref{equation_mu_h_epsilon_estimate}) we have
$|\mu_{h,\epsilon}(P)|\le 2\pi+2(c_0-1)$.
As $c_0\le 2$ (see above), thus $\mu_{h,\epsilon}(P)$ is uniformly bounded. 

In (\ref{equation_mu_h_epsilon_P}) as $\epsilon\to 0^+$,
the left side is $\mu_h(P)$,
but the right side is independent of the value of $a$;
and since $c(0)=1$ (non-zero) and $c(y)$ is continuous at 0,
the right side only depends on the value of $c(y)$ at $y=0$,
so we may replace $c(y)$ with $c(0)=1$.
Then in (\ref{equation_mu_h_half_circle}) as $\epsilon\to 0^+$,
we have
\[\mu_h(P)=\lim_{\epsilon\to 0^+}\mu_{h,\epsilon}(P)
=\lim_{\epsilon\to 0^+}\int_{-a}^a \frac{dy}{y-\epsilon i}=\pi i.
\]
So $\mu_{h,\epsilon}(P)$ is uniformly bounded and $\mu_h(P)=\pi i$.

We next consider $\mu'_h(P)$ and $\mu'_{h,\epsilon}(P)$.
By (\ref{equation_volume_P_hemisphere_r_epsilon_variant}) (with $r=1$) we have
\[\mu'_{h,\epsilon}(P)
=\int_{P\subset \mathbb{S}^1}\pm\frac{dx_1}{(x_0-\epsilon i)^2}  
=\int_{P\subset \mathbb{S}^1}\pm\frac{x_0^2}{(x_0-\epsilon i)^2}\frac{dx_1}{x_0^2}.
\]
Again $(g^{-1})^{\ast}$ maps the volume element $\pm\frac{dx_1}{x_0^2}$
in the hemisphere model (on $\mathbb{S}^1$)
into the volume element $\frac{dy}{y}$ in the upper half-space model, thus
\begin{equation}
\label{equation_mu_h_half_circle_variant}
\mu'_{h,\epsilon}(P)
=\int_{g(P)\subset \mathbb{R}}\frac{x_0^2}{(x_0-\epsilon i)^2}\frac{dy}{y}
=\int_{g(P)\subset \mathbb{R}}\frac{y}{(y-\frac{y}{x_0}\epsilon i)^2}dy.
\end{equation}

Use the same notations of $-a\le y \le a$ and $c(y)=\frac{y}{x_0}$ as above, 
as $c(y)=c(-y)$,
\begin{equation}
\label{equation_mu_h_epsilon_P_variant}
\mu'_{h,\epsilon}(P)
=\int_0^a\left( \frac{y}{(y-c(y)\epsilon i)^2} - \frac{y}{(y+c(y)\epsilon i)^2} \right) dy   \\
=\int_0^a\frac{4y^2 c(y)\epsilon i}{(y^2+c(y)^2\epsilon^2)^2} dy.   \\
\end{equation}
So for $\epsilon>0$
\begin{equation}
\label{equation_mu_h_epsilon_estimate_variant}
|\mu'_{h,\epsilon}(P)|
\le \int_0^a\frac{4c(y)\epsilon}{y^2+c(y)^2\epsilon^2} dy
=2\int_0^a b(y,c(y)\epsilon) dy.
\end{equation}
By (\ref{equation_mu_h_epsilon_estimate}) we have
$|\mu'_{h,\epsilon}(P)|\le 2|\mu_{h,\epsilon}(P)|$.
As $\mu_{h,\epsilon}(P)$ is uniformly bounded,
so $\mu'_{h,\epsilon}(P)$ is uniformly bounded as well.

In (\ref{equation_mu_h_epsilon_P_variant}) as $\epsilon\to 0^+$,
the left side is $\mu'_h(P)$ but the right side is independent of the value of $a$.
As before, to compute $\mu'_h(P)$, 
in (\ref{equation_mu_h_half_circle_variant}) we replace $\frac{y}{x_0}$ with 1, then

\begin{align*}
\mu'_h(P)
&=\lim_{\epsilon\to 0^+} \mu'_{h,\epsilon}(P)    
=\lim_{\epsilon\to 0^+} \int_{-a}^a \frac{y}{(y-\epsilon i)^2}dy          \\
&=\lim_{\epsilon\to 0^+} -\int_{-a}^a y d\left(\frac{1}{y-\epsilon i}\right)     
=\lim_{\epsilon\to 0^+} \int_{-a}^a \frac{dy}{y-\epsilon i}     
=\pi i.
\end{align*}
So $\mu'_{h,\epsilon}(P)$ is uniformly bounded and $\mu'_h(P)=\pi i$.

Finally,  
$\mu(P)$, $\mu_h(P)$, and $\mu'_h(P)$ all agree with the same value $\pi i$,
so they are invariant under isometry.
Thus Theorem~\ref{theorem_volume_invariant} is true for dimension 1.
\end{proof}

\section{Special cases of \SDF{} for $\mathbb{DH}^n$}
\label{section_schlafli_special}

We first clarify some notions by pointing out some combinatorial differences between
$\mathbb{DH}^n$ and $\mathbb{H}^n$.
Recall that a polytope $P$ in $\mathbb{DH}^n$
is a finite intersection of closed half-spaces in $\mathbb{DH}^n$.
A \emph{face} of $P$ is an intersection of $P$ and the faces of some of the closed half-spaces.
(Due to the symmetry between $\mathbb{H}^n$ and $\mathbb{H}^n_{-}$, 
a 0-dimensional face is always a $\mathbb{DH}^0$ that contains a pair of vertices.)
When the intersection is a point on $\partial\mathbb{H}^n$, it is called an \emph{ideal vertex}, see Figure~\ref{figure_ideal}.
In $\mathbb{H}^n$ if we relax the convex polytope to include ideal vertices, 
then an ideal vertex is also a 0-dimensional face of the convex polytope.
But in $\mathbb{DH}^n$ we do not treat ideal vertex as a face of $P$,
because an ideal vertex is not a vertex of any 1-dimensional face of $P$.
Any face of $P$ in $\mathbb{DH}^n$ is a lower dimensional polytope itself.

\begin{figure}[h]
\centering
  \includegraphics[width=0.4\textwidth]{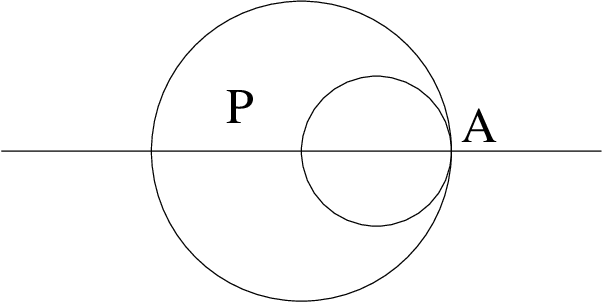}
\caption{In upper half-space model, $A$ is an ideal vertex of polytope $P$}
\label{figure_ideal}
\end{figure}

We next prove Theorem~\ref{theorem_volume_invariant} 
(with dimension 0 and 1 verified in Lemma~\ref{lemma_volume_invariant_0_1})
and \ref{theorem_schlafli} together by induction.
We start with some special cases of the \SDF{} for $\mathbb{DH}^n$.

\subsection{First special case}
\label{section_special_first}

Assuming Theorem~\ref{theorem_volume_invariant} is true 
up to dimension $n-1\ge 1$, then for dimension $n$,
we first prove a special case of \SDF{} of $\mu(P)$
in the upper half-space model.
Let $P_t$ be a family of polytopes in $\mathbb{DH}^n$
with a single parameter $t$ that is the intersection of 
a fixed $P$ with a moving half-space $x_{n-1}\le t$.
Here we allow $P$ to be any polytope in $\mathbb{DH}^n$ and with ideal vertices allowed,
but assume that $P$ is in a finite region in $\mathbb{R}^n$
that satisfies $t_0\le x_{n-1}\le t_1$.
The proof runs inductively on $n\geq 2$ with $V_0(\mathbb{DH}^0)=2$
and $V_1(\mathbb{DH}^1)=2\pi i$ as given (see (\ref{equation_total_volume_dh1})).

By (\ref{equation_volume_P_half_space_epsilon}) we have
\[\mu_{\epsilon}(P_t)= \int_{P_t\subset\mathbb{R}^n}\frac{dx_0\cdots dx_{n-1}}{(x_0-\epsilon i)^n},
\quad\mu(P_t)=\lim_{\epsilon\to 0^+} \mu_{\epsilon}(P_t).
\]
If $E_t$ is the facet (which may possibly contain two connected pieces)
of $P_t$ in the plane $x_{n-1}=t$, let
\begin{equation}
\label{equation_derivative_u_epsilon_t}
f_{\epsilon}(t):=\int_{E_t}\frac{dx_0\cdots dx_{n-2}}{(x_0-\epsilon i)^n},
\end{equation}
then observe that  
\begin{equation}
\label{equation_volume_epsilon_integration_derivative}
\mu_{\epsilon}(P_t)=\int_{t_0}^{t}f_{\epsilon}(t)dt
\quad\text{and}\quad
\frac{d\mu_{\epsilon}(P_t)}{dt}=f_{\epsilon}(t).
\end{equation}
On the right side of (\ref{equation_derivative_u_epsilon_t}),
integrating with respect to $x_0$, we obtain
\begin{equation}
\label{equation_volume_derivative_epsilon}
f_{\epsilon}(t)
=-\frac{1}{n-1}\int_{\partial E_t}\pm\frac{dx_1\cdots dx_{n-2}}{(x_0-\epsilon i)^{n-1}}.
\end{equation}
We next explain the ``$\pm$'' in (\ref{equation_volume_derivative_epsilon}).

Recall that in the upper half-space model
a half-space is either the inside or the outside of a ball, 
or a Euclidean half-space whose face is a vertical plane to $x_0=0$.
Respectively, we call the corresponding face 
\emph{top}, \emph{bottom}, or \emph{side},
for both the parts in the upper and lower half-spaces
(see Figure~\ref{figure_faces}).
In $E_t$, for a top $(n-2)$-dimensional face $F$,
the part in the upper half-space is counted positively
in (\ref{equation_volume_derivative_epsilon}),
and in the lower half-space is counted negatively.
For a bottom face $F$ of $E_t$,
the part in the upper half-space is counted negatively,
and in the lower half-space is counted positively.
The side faces of $E_t$, where $dx_1\cdots dx_{n-2}$ is 0,
do not contribute to (\ref{equation_volume_derivative_epsilon}).

\begin{figure}[h]
\centering
  \includegraphics[width=0.3\textwidth]{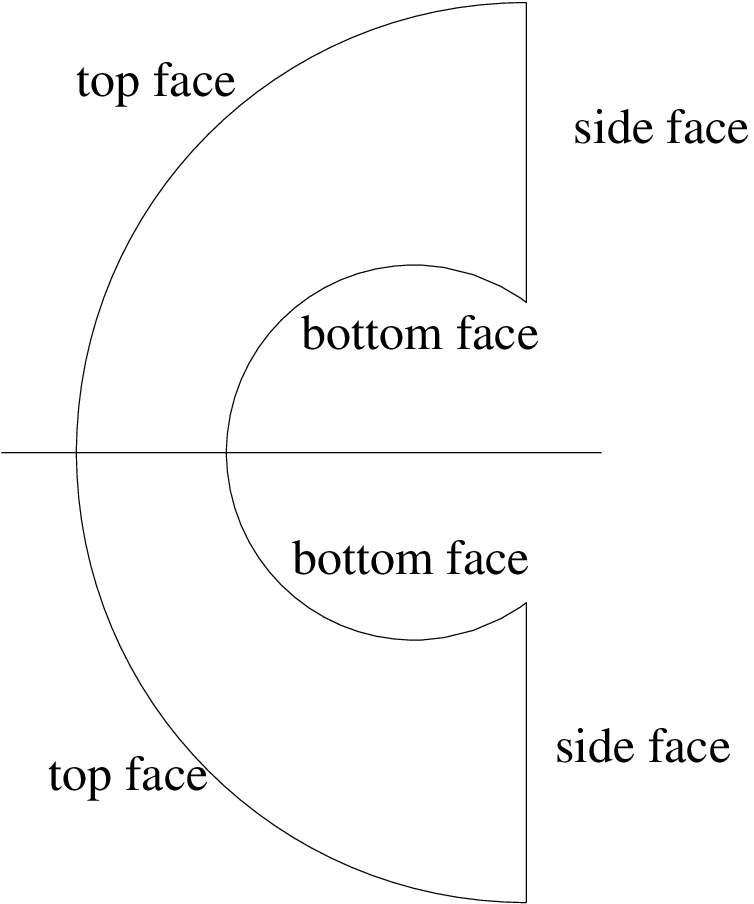}
\caption{The top, bottom, and side faces}
\label{figure_faces}
\end{figure}

Next we show that the contribution of any $(n-2)$-dimensional face $F$
to the integral in (\ref{equation_volume_derivative_epsilon}) is 
$\mu'_{h,\epsilon}(F)\frac{d\theta_F}{dt}$,
where $\mu'_{h,\epsilon}(F)$ is defined in (\ref{equation_volume_P_hemisphere_r_epsilon_variant})
and $\theta_F$ is the dihedral angle of $P_t$ at $F$. 
If $F$ is the intersection of $E_t$ and a fixed vertical plane, then $\theta_F$ is fixed over $t$
and thus as expected $\frac{d\theta_F}{dt}=0$.
Otherwise,  $F$ is the intersection of $E_t$ and a fixed $(n-1)$-sphere $E'$ with radius $r$,
where $F$ is on an $(n-2)$-sphere with radius $r_{\sss F}$ (see Figure~\ref{figure_angle}).

\begin{figure}[h]
\centering
\resizebox{.4\textwidth}{!}
  {\input{fig_angle.pspdftex}}
\caption{The dihedral angle $\theta_F$ at the intersection of $E_t$ and $E'$, assuming $E'$ is a top face of $P_t$}
\label{figure_angle}
\end{figure}

In (\ref{equation_volume_derivative_epsilon}),
by using (\ref{equation_volume_P_hemisphere_r_epsilon_variant}) for non-unit spheres, we have
\begin{equation}
\label{equation_volume_epsilon}
f_{\epsilon}(t)
=-\frac{1}{n-1}\sum_{F}\pm\frac{1}{r_{\sss F}}\mu'_{h,\epsilon}(F),
\end{equation}
with the plus sign for a top face and the minus sign for a bottom face,
where only those $(n-2)$-faces $F$ on $E_t$ in $x_{n-1}=t$ contribute to the formula.

By Figure~\ref{figure_angle}, it is easy to check that $\theta_F$ satisfies
\begin{equation}
\label{equation_theta_F_t}
\sin\theta_F=\frac{r_{\sss F}}{r},\quad -r\cos\theta_F=\pm(t-c),
\end{equation}
where $c$ is the coordinate $x_{n-1}$ of the center of $E'$,
with the plus sign if $E'$ is a top face of $P$ and the minus sign if $E'$ is a bottom face.
Differentiating the right hand equation with respect to $t$,
we obtain $r\sin\theta_F\cdot \frac{d\theta_F}{dt}=\pm 1$.
By the left hand equation, therefore $\frac{d\theta_F}{dt}=\pm \frac{1}{r_{\sss F}}$.
In (\ref{equation_volume_epsilon}) replacing $\pm \frac{1}{r_{\sss F}}$ with $\frac{d\theta_F}{dt}$, we have
\begin{equation}
\label{equation_schlafli_special_epsilon}
f_{\epsilon}(t)
=-\frac{1}{n-1}\sum_{F}\mu'_{h,\epsilon}(F)\,\frac{d\theta_F}{dt}.
\end{equation}
So by (\ref{equation_volume_epsilon_integration_derivative})
\begin{equation}
\label{equation_schlafli_special_derivative_epsilon}
\frac{d\mu_{\epsilon}(P_t)}{dt}
=-\frac{1}{n-1}\sum_{F}\mu'_{h,\epsilon}(F)\,\frac{d\theta_F}{dt}.
\end{equation}

Notice that $\frac{d\theta_F}{dt}$, whose value is $\pm \frac{1}{r_{\sss F}}$ above,
depends only on the value of $r_{\sss F}$ but not $r$ (the radius of $E'$),
though $\theta_F$ depends on both.
This is expected, as for a \emph{fixed} $t$, $f_{\epsilon}(t)$ depends on $F$
but not on how $F$ is fitted in $E'$;
particularly, $f_{\epsilon}(t)$ depends on $r_{\sss F}$ but not on $r$.
In fact, we have the following general statement.

\begin{lemma}
\label{lemma_differential_theta_indepentent}
For a polytope $P$ in $\mathbb{DH}^n$,
let $F$ be the intersection of two $(n-1)$-faces $E$ and $E'$, and $\theta_F$ be the dihedral angle.
If $E'$ is fixed but $E$ can have any small movement, then for a fixed $F$, 
$d\theta_F$ does not depend on (the choice of) $E'$.
\end{lemma}

\begin{proof}
This property does not depend on the model we use, for our convenience  we use the hyperboloid model.
In the upper sheet, let $u_{\sss E,P}$ be the inward unit normal to $P$ at its face $E$,
$u_{\sss F,E}$ be the inward unit normal to $E$ at its face $F$, and so on.
It is easy to check that
\[d\theta_F=u_{\sss F,E}\cdot du_{\sss E,P}+u_{\sss F,E'}\cdot du_{\sss E',P},
\]
a formula used by Alexander \cite{Alexander:Lipschitzian} to give a proof 
of the \SDF{} for the Euclidean case.
As $E'$ is fixed, so $du_{\sss E',P}=0$, thus $d\theta_F=u_{\sss F,E}\cdot du_{\sss E,P}$,
which does not depend on the choice of $E'$.
\end{proof}

Let $f(t)$ be the pointwise limit of $f_{\epsilon}(t)$ as $\epsilon\to 0^{+}$,
\begin{equation}
\label{equation_derivative_special_limit}
f(t)=\lim_{\epsilon\to 0^+} f_{\epsilon}(t),
\end{equation}
then by (\ref{equation_schlafli_special_epsilon}) we immediately have the following.

\begin{lemma}
Assuming Theorem~\ref{theorem_volume_invariant} is true up to dimension $n-1\ge 1$, 
then
\begin{equation}
\label{equation_schlafli_special}
f(t)=-\frac{1}{n-1}\sum_{F} V_{n-2}(F)\,\frac{d\theta_F}{dt},
\end{equation}
and $f(t)$ is continuous except at a finite number of points where $P_t$ changes its combinatorial type.
\end{lemma}

We remark that $f(t)$ may also not be bounded at those points when $\theta_F=0$ or $\pi$ where $r_{\sss F}=0$,
but we will show that $f(t)$ is still integrable.
We have the following special case of \SDF{} of $\mu(P)$.

\begin{lemma}
\label{lemma_SDF_special}
Assuming Theorem~\ref{theorem_volume_invariant} is true up to dimension $n-1\ge 1$.
For a fixed $m$, let $P$ be a polytope in $\mathbb{DH}^n$ 
and be the intersection of at most $m$ half-spaces.
In the upper half-space model 
let $P_t$ be the intersection of $P$ and $x_{n-1}\le t$,
and assume $P$ satisfies $t_0\le x_{n-1}\le t_1$.
\begin{enumerate}
\item[(1)]
Then $f(t)$ is integrable,
$\mu(P)$ exists and $\mu(P_t)$ is continuous for $t$ with $\mu(P_t)=\int_{t_0}^t f(t)dt$, 
and
\begin{equation}
\label{equation_SDF_special}
\kappa\cdot\frac{d\mu(P_t)}{dt}=\frac{1}{n-1}\sum_{F} V_{n-2}(F)\,\frac{d\theta_F}{dt},
\end{equation}
where the sum is taken over all $(n-2)$-faces $F$ of $P_t$ on $x_{n-1}=t$.
\item[(2)]
For a fixed $m$, $\mu_{\epsilon}(P)$ is uniformly bounded for all $P$ and $\epsilon>0$.
\end{enumerate}
\end{lemma}

\begin{proof}
We first prove (1). 
As it is assumed that Theorem~\ref{theorem_volume_invariant} is true for dimension $n-2$,
there is a constant $c>0$, depending only on $m$ but not $P$ and $\epsilon$,
such that $|\mu'_{h,\epsilon}(F)|\le c$ for all $(n-2)$-faces $F$ of $P_t$.
By (\ref{equation_schlafli_special_epsilon}), we have 
\[|f_{\epsilon}(t)|\le \frac{c}{n-1}\sum_{F}\left|\frac{d\theta_F}{dt}\right|.
\]
For any non-empty $(n-2)$-face $F$, as $\theta_F$ is a monotonic function of $t$
that takes values in an interval between 0 and $\pi$, 
therefore in this interval $|\frac{d\theta_F}{dt}|$ is integrable and
\[\int \left|\frac{d\theta_F}{dt}\right| dt\le\pi.
\]
Set $g(t)=\frac{c}{n-1}\sum_{F}|\frac{d\theta_F}{dt}|$, then  $|f_{\epsilon}(t)| \le g(t)$.
As there are at most $m$ half-spaces, so there are at most $m$ such $(n-2)$-faces $F$ of $P_t$ on $x_{n-1}=t$,
thus we have
\[\int g(t)\,dt \le \frac{c m\pi}{n-1}<\infty.
\]
Therefore Lebesgue dominated convergence theorem applies to $\{f_{\epsilon}\}$ 
and $f$ is integrable, 
and by (\ref{equation_volume_epsilon_integration_derivative})
and (\ref{equation_derivative_special_limit})
\[\mu(P_t)=\lim_{\epsilon\to 0^+} \mu_{\epsilon}(P_t)
=\lim_{\epsilon\to 0^+}\int_{t_0}^t f_{\epsilon}(t)dt=\int_{t_0}^t f(t)dt.
\]
So $\mu(P)$ exists and $\mu(P_t)$ is continuous for $t$
with $\mu(P_t)=\int_{t_0}^t f(t)dt$.
Then $\frac{d\mu(P_t)}{dt}=f(t)$ except at 
a finite number of points where $P_t$ changes its combinatorial type
(and $f(t)$ may not be continuous or even bounded).
Then in (\ref{equation_schlafli_special}) multiplying by $\kappa=-1$,
we prove (\ref{equation_SDF_special}). This finishes the proof of (1).

By (\ref{equation_volume_epsilon_integration_derivative}) we also have
$|\mu_{\epsilon}(P)| \le \int|f_{\epsilon}|dt \le \int g dt \le \frac{c m\pi }{n-1}$.
So for a fixed $m$, $\mu_{\epsilon}(P)$ is uniformly bounded.
This proves (2) and finishes the proof.
\end{proof}

\begin{remark}
Due to the rotational symmetry of $\mu_{\epsilon}(P)$, 
Lemma~\ref{lemma_SDF_special} may apply 
to any vertical plane to $x_0=0$ that sweeps through $P$ parallelly,
not just for plane $x_{n-1}=t$.
\end{remark}

We have the following result that will be useful for later use,
for both the hemisphere model and the upper half-space model.

\begin{lemma}
\label{lemma_f_u_epsilon_limit}
Assuming Theorem~\ref{theorem_volume_invariant} is true up to dimension $n-1\ge 1$.
For a fixed $m$, let $E$ be an $(n-1)$-dimensional polytope (in $\mathbb{DH}^{n-1}$)
on $S^{n-1}_r$ and be the intersection of at most $m$ half-spaces. 
Then $\frac{\epsilon}{r}\mu'_{h,\epsilon}(E)$ is uniformly bounded for all $r$, $E$ and $\epsilon>0$.
\end{lemma}

\begin{proof}
By (\ref{equation_volume_P_hemisphere_r_epsilon_variant}) we have

\begin{align*}
\left|\frac{\epsilon}{r} \mu'_{h,\epsilon}(E)\right| 
&= \left| \int_{E\subset S^{n-1}_r} \pm\frac{\epsilon dx_1\cdots dx_{n-1}}{(x_0-\epsilon i)^n} \right|
\le \frac{\epsilon}{\epsilon^n} \int_{E\subset S^{n-1}_r} dx_1\cdots dx_{n-1}      \\
&\le \frac{r^{n-1}}{\epsilon^{n-1}} \int_{\mathbb{S}^{n-1}} dx_1\cdots dx_{n-1}     
=  \frac{r^{n-1}}{\epsilon^{n-1}}\cdot 2V_{n-1}(\mathbb{B}^{n-1}),
\end{align*}
where $V_{n-1}(\mathbb{B}^{n-1})$ is the Euclidean volume of the unit $(n-1)$-ball $\mathbb{B}^{n-1}$,
and $V_0(\mathbb{B}^0)=1$. Denote $2 V_{n-1}(\mathbb{B}^{n-1})$ by $c_1$.
On the other hand, 
as  it is assumed that Theorem~\ref{theorem_volume_invariant} is true for dimension $n-1$,
there is a constant $c_2>0$ such that 
 $|\mu'_{h,\epsilon}(E)| \le c_2$ for all $E$ and $\epsilon>0$.
 So
 \[\left|\frac{\epsilon}{r} \mu'_{h,\epsilon}(E)\right| 
 \le \min\left\{c_1\frac{r^{n-1}}{\epsilon^{n-1}}, c_2\frac{\epsilon}{r} \right\}
 \le \left(c_1\frac{r^{n-1}}{\epsilon^{n-1}}\right)^{\frac{1}{n}} \cdot \left(c_2\frac{\epsilon}{r}\right)^{\frac{n-1}{n}}
 =c_1^\frac{1}{n}\cdot c_2^{\frac{n-1}{n}},
 \]
where the second step is the weighted geometric mean of 
$c_1\frac{r^{n-1}}{\epsilon^{n-1}}$ and $c_2\frac{\epsilon}{r}$
with weights $\frac{1}{n}$ and $\frac{n-1}{n}$.
Then for a fixed $m$, $\frac{\epsilon}{r} \mu'_{h,\epsilon}(E)$ is uniformly bounded.
\end{proof}

\subsection{Second special case}
\label{section_special_second}

Before proving that $\mu(P)$ is invariant under isometry,
we show that $\mu(P)$ also satisfies a special \SDF{} in the radial direction.
For $n\ge 2$, let $B_r$ be a ball centered at $O$ with radius $r>0$ where 
$r^2=x_0^2+\cdots+x_{n-1}^2$, and $S^{n-1}_r$ be the spherical boundary of $B_r$.
Denote the intersection of $P$ and $B_r$ by $P'_r$.
We have the following important generalization of (\ref{equation_schlafli_special_derivative_epsilon}).

\begin{lemma}
\label{lemma_schlafli_epsilon}
For $n\ge 2$, let $E$ be the $(n-1)$-face of $P'_r$ on $S^{n-1}_r$. Then
\begin{equation}
\label{equation_schlafli_epsilon}
\frac{d\mu_{\epsilon}(P'_r)}{dr}
=-\frac{1}{n-1}\sum_{F}\mu'_{h,\epsilon}(F)\,\frac{d\theta_F}{dr}
+\frac{\epsilon i}{r^2}\mu'_{h,\epsilon}(E),
\end{equation}
where  the sum is taken over all $(n-2)$-faces $F$ on $E$,
with $\theta_F$ the dihedral angle.
\end{lemma}

\begin{proof}
We only need to prove the case at a fixed $r=r_0$ when $E$ is non-empty.
Let $F_i$ be the intersection of $E$ and a fixed $(n-1)$-face $E_i$,
then $F_i$ lies on the intersection of $E$ and a \emph{vertical} plane $H_i$ to $x_0=0$
(see Figure~\ref{figure_polar}).
At a fixed $r=r_0$, by Lemma~\ref{lemma_differential_theta_indepentent},
the value of $\frac{d\theta_{F_i}}{dr}$ does not depend on how $F_i$ is fitted in $E_i$,
and thus $\frac{d\theta_{F_i}}{dr}$ can be computed as if $E_i$ is replaced by a fixed $H_i$;
and for a fixed $\epsilon>0$, because the coefficient of the volume element of $\mu_{\epsilon}$ 
is $\frac{1}{(x_0-\epsilon i)^n}$ and bounded by $\frac{1}{\epsilon^n}$ from above, so
$\left.\frac{d\mu_{\epsilon}(P'_r)}{dr}\right\vert_{r=r_0}$ does not depend on how $F_i$ is fitted in $E_i$ either.
So to prove (\ref{equation_schlafli_epsilon}), without loss of generality we assume that for $P'_r$,
except for $E$, all remaining $(n-1)$-faces are side faces on fixed $H_i$.

\begin{figure}[h]
\centering
\resizebox{.4\textwidth}{!}
  {\input{fig_polar.pspdftex}}
\caption{$E_i$ is replaced by a fixed vertical plane $H_i$}
\label{figure_polar}
\end{figure}

When $r$ varies at $r_0$, for $P'_r$, $E$ varies but $H_i$ are fixed.
By (\ref{equation_volume_P_half_space_epsilon}), 
notice that $\mu_{\epsilon}(P'_r)$ is ``almost'' invariant under \emph{similarity},
in the sense that only $\epsilon$ needs to be adjusted properly.
This property suggests that,
by a composition of the following two steps,
$\left.\frac{d\mu_{\epsilon}(P'_r)}{dr}\right\vert_{r=r_0}$ can be computed 
as if $E$ is fixed but $H_i$ are moving \emph{parallelly}
(at paces that preserve $\frac{d\theta_{F_i}}{dr}$ as before),
but $\epsilon$ needs to be adjusted to $\frac{r_0}{r}\epsilon$ over $r$.
When varying $H_i$ parallelly the dihedral angles between the side faces do not change, 
so by applying (\ref{equation_schlafli_special_derivative_epsilon})
to all $H_i$ with the adjustment of $\epsilon$ to $P'_{r_0}$,
we have
\begin{equation}
\label{equation_derivative_volume_P'_epsilon_t}
\left.\frac{d\mu_{\epsilon}(P'_r)}{dr}\right\vert_{r=r_0}
=-\frac{1}{n-1}\sum_{F_i} \mu'_{h,\epsilon}(F_i)\,\left.\frac{d\theta_{F_i}}{dr}\right\vert_{r=r_0}
+\left.\frac{d\mu_{r_0\epsilon/r}(P'_{r_0})}{dr}\right\vert_{r=r_0}.
\end{equation}

By (\ref{equation_volume_P_half_space_epsilon}), we have


\[\left.\frac{d\mu_{r_0\epsilon/r}(P'_{r_0})}{dr}\right\vert_{r=r_0}
=\left. -n\frac{r_0\epsilon i}{r^2}\int_{P'_{r_0}}\frac{dx_0\cdots dx_{n-1}}{(x_0-\frac{r_0}{r}\epsilon i)^{n+1}} \right\vert_{r=r_0}.
\]
Denote $E$ at $r=r_0$ by $E_0$, the above becomes
\[-n\frac{\epsilon i}{r_0^2}\int_{P'_{r_0}}\frac{r_0 dx_0\cdots dx_{n-1}}{(x_0-\epsilon i)^{n+1}}
=\frac{\epsilon i}{r_0^2}\int_{E_0} \pm\frac{r_0 dx_1\cdots dx_{n-1}}{(x_0-\epsilon i)^n}
=\frac{\epsilon i}{r_0^2}\mu'_{h,\epsilon}(E_0),
\]
where the first step is by integrating with respect to $x_0$,
and the last step is by applying (\ref{equation_volume_P_hemisphere_r_epsilon_variant}) to $E_0$.
Plug it into (\ref{equation_derivative_volume_P'_epsilon_t}),
we then prove (\ref{equation_schlafli_epsilon}) at $r=r_0$.
\end{proof}

Let 
\begin{equation}
\label{equation_derivative_u_epsilon_r}
a_{\epsilon}(r)
:=\int_{P\cap S^{n-1}_r}\frac{dA_r}{(x_0-\epsilon i)^n},
\end{equation}
where $dA_r$ is the Euclidean volume element on $(n-1)$-sphere $S^{n-1}_r$.
Then by using polar coordinates to $\mu_{\epsilon}(P'_r)$, for any $r_0>0$ we have
\begin{equation}
\label{equation_volume_epsilon_integration_derivative_2}
\mu_{\epsilon}(P'_r)-\mu_{\epsilon}(P'_{r_0})
=\int_{r_0}^{r} a_{\epsilon}(r)dr
\quad\text{and}\quad
\frac{d\mu_{\epsilon}(P'_r)}{dr}=a_{\epsilon}(r).
\end{equation}
Adding this $r_0>0$ is for handling convergence issues that may arise at $r=0$ at later moments.
So $a_{\epsilon}(r)$ is the derivative of $\mu_{\epsilon}(P)$ with respect to $r$ in the radial direction,
while $f_{\epsilon}(t)$ in a side direction (see (\ref{equation_volume_epsilon_integration_derivative})).
By Lemma~\ref{lemma_schlafli_epsilon} and (\ref{equation_volume_epsilon_integration_derivative_2}),
we immediately have the following.

\begin{lemma}
\label{lemma_schlafli_special_epsilon_2}
For $n\ge 2$ and $r>0$, 
\begin{equation}
\label{equation_schlafli_special_epsilon_2}
a_{\epsilon}(r)
=-\frac{1}{n-1}\sum_{F}\mu'_{h,\epsilon}(F)\,\frac{d\theta_F}{dr}
+\frac{\epsilon i}{r^2}\mu'_{h,\epsilon}(E),
\end{equation}
where $E=P\cap S^{n-1}_r$, 
and the sum is taken over all $(n-2)$-faces $F$ on $E$,
with $\theta_F$ the dihedral angle at $F$.
\end{lemma}

In (\ref{equation_schlafli_special_epsilon_2}),
for $\frac{d\theta_F}{dr}$ we now provide an explicit formula.
We first consider a special case that $\theta_F$ is the dihedral angle 
between $S^{n-1}_r$ and a fixed side face.

\begin{lemma}
\label{lemma_theta_derivative_r}
Let $\theta_F$ be the dihedral angle between $S^{n-1}_r$ and a side face $H$
of a fixed half-space in $\mathbb{DH}^n$, and $r_{\sss F}$ be the radius of $F$.
Then
\begin{equation}
\label{equation_theta_derivative_r}
\frac{d\theta_F}{dr}
=\pm\frac{(r^2-r_{\sss F}^2)^{1/2}}{r\cdot r_{\sss F}},
\end{equation}
taking the plus sign or the minus sign according as the origin $O$ 
lies outside or inside of the fixed half-space.
If $O$ is on $H$, then $r_{\sss F}=r$ and $\theta_F$ is a constant $\pi/2$.
\end{lemma}

\begin{proof}
Denote $\pm (r^2-r_{\sss F}^2)^{1/2}$ by $d$, the signed Euclidean distance from $O$ to $H$,
taking the plus sign or the minus sign according as the origin $O$ lies outside or inside of the fixed half-space.
As $H$ is fixed, so $d$ is a constant and $\theta_F$ satisfies 
(see also Figure~\ref{figure_angle}, but be aware of the notational difference)
\[\sin\theta_F=\frac{r_{\sss F}}{r},\quad r\cos\theta_F=d.
\]
In the right hand equation, on both sides first divide by $r$ and then differentiate with respect to $r$, we have
\[-\sin\theta_F\frac{d\theta_F}{dr}=-\frac{d}{r^2}.
\]
So combined with the left hand equation,
\[\frac{d\theta_F}{dr}=\frac{d}{r^2\sin\theta_F}
=\frac{d}{r\cdot r_{\sss F}}
=\pm\frac{(r^2-r_{\sss F}^2)^{1/2}}{r\cdot r_{\sss F}}.
\]
This finishes the proof.
\end{proof}

We next consider the general case that $\theta_F$ is the dihedral angle 
between $S^{n-1}_r$ and a non-side face $E'$.
By Lemma~\ref{lemma_differential_theta_indepentent},
we immediately have the following.

\begin{corollary}
\label{corollary_theta_derivative_r}
Let $\theta_F$ be the dihedral angle between $S^{n-1}_r$ and a fixed $(n-1)$-face $E'$
(not necessarily a side face), and $r_{\sss F}$ be the radius of $F$. 
Then
$\frac{d\theta_F}{dr}=\pm\frac{(r^2-r_{\sss F}^2)^{1/2}}{r\cdot r_{\sss F}}$,
and by Lemma~\ref{lemma_schlafli_special_epsilon_2}
\begin{equation}
\label{equation_schlafli_special_epsilon_2_formula}
a_{\epsilon}(r)
=-\frac{1}{n-1}\sum_{F}\pm \frac{(r^2-r_{\sss F}^2)^{1/2}}{r\cdot r_{\sss F}} \mu'_{h,\epsilon}(F)
+\frac{\epsilon i}{r^2}\mu'_{h,\epsilon}(E),
\end{equation}
where $E=P\cap S^{n-1}_r$
and the ``$\pm$'' is explained in Lemma~\ref{lemma_theta_derivative_r}.
\end{corollary}

\begin{remark}
\label{remark_theta_derivative_r}
In Corollary~\ref{corollary_theta_derivative_r}, while $\theta_F$ may not be monotonic over $r$,
$\frac{d\theta_F}{dr}$ will change sign at most once, which happens when $r=r_{\sss F}$.
\end{remark}

For $r>0$, let $a(r)$ be the pointwise limit of $a_{\epsilon}(r)$ as $\epsilon\to 0^+$,
\begin{equation}
\label{equation_derivative_special_limit_2}
a(r)=\lim_{\epsilon\to 0^+} a_{\epsilon}(r),
\end{equation}
then we have the following.

\begin{lemma}
Assuming Theorem~\ref{theorem_volume_invariant} is true up to dimension $n-1\ge 1$, 
then
\begin{equation}
\label{equation_schlafli_special_2}
a(r) = -\frac{1}{n-1}\sum_{F} V_{n-2}(F)\,\frac{d\theta_F}{dr},
\end{equation}
and $a(r)$ is continuous for $r>0$
except at a finite number of points where $P'_r$ changes its combinatorial type. 
\end{lemma}

\begin{proof}
As it is assumed that Theorem~\ref{theorem_volume_invariant} is true up to dimension $n-1\ge 1$, 
in (\ref{equation_schlafli_special_epsilon_2})
$\mu'_{h,\epsilon}(E)$ is uniformly bounded for all $\epsilon>0$,
so $\lim_{\epsilon\to 0^+} \epsilon\mu'_{h,\epsilon}(E)=0$.
Then in (\ref{equation_schlafli_special_epsilon_2})
by taking the limits on both sides as $\epsilon\to 0^+$, we prove (\ref{equation_schlafli_special_2}).
\end{proof}

We next prove a convergence property of $a_{\epsilon}(r)$,
which will be useful for both the upper half-space model and the hemisphere model.

\begin{lemma}
\label{lemma_a_epsilon_limit_convergent}
Assuming Theorem~\ref{theorem_volume_invariant} is true up to dimension $n-1\ge 1$.
For a fixed $m$,
let $P$ be a polytope in $\mathbb{DH}^n$ and be the intersection of at most $m$ half-spaces.
Then for any $r_0$ and $r_1$ with $0<r_0<r_1$,
the sequence $\{a_{\epsilon}(r)\}$ is dominated by an integrable function $g(r)$ on $[r_0,r_1]$
with $|a_{\epsilon}(r)|\le g(r)$ for all $\epsilon>0$.
Thus Lebesgue dominated convergence theorem applies 
and $a(r)$ is integrable on $[r_0,r_1]$, and
\begin{equation}
\label{equation_a_epsilon_limit_convergent}
\lim_{\epsilon\to 0^+}\int_{r_0}^{r_1} a_{\epsilon}(r)dr=\int_{r_0}^{r_1} a(r)dr.
\end{equation}
Also for fixed m, $r_0$ and $r_1$,
$\int_{r_0}^{r_1} a_{\epsilon}(r)dr$ is uniformly bounded for all $P$ and $\epsilon>0$.
\end{lemma}

\begin{proof}
As it is assumed that Theorem~\ref{theorem_volume_invariant} is true up to dimension $n-1\ge 1$,
there is a constant $c_1$, depending only on $m$ but not on $P$ and $\epsilon$,
such that for all $(n-2)$-faces $F$ of $P'_r$, we have $|\mu'_{h,\epsilon}(F)|\le c_1$.
Let $E=P\cap S^{n-1}_r$, so $r$ is the radius of $E$.
Then by Lemma~\ref{lemma_f_u_epsilon_limit}, 
there is also a constant $c_2$ that depends only on $m$,
such that $|\frac{\epsilon}{r}\mu'_{h,\epsilon}(E)|\le c_2$.
By Lemma~\ref{lemma_schlafli_special_epsilon_2} we have
\begin{equation}
\label{equation_schlafli_special_epsilon_2_bound}
|a_{\epsilon}(r)|
\le \frac{c_1}{n-1}\sum_F \left|\frac{d\theta_F}{dr}\right| + \frac{c_2}{r}.
\end{equation}
Set $g(r)=\frac{c_1}{n-1}\sum_F |\frac{d\theta_F}{dr}| + \frac{c_2}{r}$,
then $|a_{\epsilon}(r)|\le g(r)$.
By Corollary~\ref{corollary_theta_derivative_r} and Remark~\ref{remark_theta_derivative_r}, 
$\frac{d\theta_F}{dr}$ will change sign at most once.
So on any interval $[r_0,r_1]$ with $r_0>0$,
$|\frac{d\theta_F}{dr}|$ is integrable
and $\int_{r_0}^{r_1} |\frac{d\theta_F}{dr}|dr \le\pi+\pi=2\pi$.
As there are at most $m$ half-spaces, so there are at most
$m$ such $(n-2)$-faces $F$ of $P'_r$ on $S^{n-1}_r$,
thus $g(r)$ is also integrable on $[r_0,r_1]$.
Therefore on $[r_0,r_1]$ Lebesgue dominated convergence theorem 
applies to $\{a_{\epsilon}(r)\}$ and $a(r)$ is integrable,
and thus proves (\ref{equation_a_epsilon_limit_convergent}).

For fixed $m$, $r_0$ and $r_1$, we have $\frac{1}{r}\le \frac{1}{r_0}$ for $r_0\le r\le r_1$,
so by (\ref{equation_schlafli_special_epsilon_2_bound})
\[\left| \int_{r_0}^{r_1} a_{\epsilon}(r)dr \right|
\le \frac{2c_1 m\pi}{n-1} + \frac{c_2(r_1-r_0)}{r_0},
\]
and thus $\int_{r_0}^{r_1} a_{\epsilon}(r)dr$ is uniformly bounded.
This finishes the proof.
\end{proof}

Then we have a second special case of \SDF{} of $\mu(P)$.

\begin{lemma}
\label{lemma_SDF_special_2}
Assuming Theorem~\ref{theorem_volume_invariant} is true up to dimension $n-1\ge 1$.
Let $P$ be a polytope in $\mathbb{DH}^n$ and $P'_r=P\cap B_r$,
then 
$\mu(P'_r)$ is continuous for $r>0$ and
\begin{equation}
\label{equation_SDF_special_2}
\kappa\cdot\frac{d\mu(P'_r)}{dr}
=\frac{1}{n-1}\sum_{F} V_{n-2}(F)\,\frac{d\theta_F}{dr},
\end{equation}
where the sum is taken over all $(n-2)$-faces $F$ of $P'_r$ on $S^{n-1}_r$.
\end{lemma}

\begin{proof}
For any $r_0>0$, by (\ref{equation_volume_epsilon_integration_derivative_2})
\[\mu(P'_r)-\mu(P'_{r_0})
=\lim_{\epsilon\to 0^+} (\mu_{\epsilon}(P'_r)- \mu_{\epsilon}(P'_{r_0}))
=\lim_{\epsilon\to 0^+}\int_{r_0}^r a_{\epsilon}(r)dr.
\]
By Lemma~\ref{lemma_a_epsilon_limit_convergent} we have
\[\lim_{\epsilon\to 0^+}\int_{r_0}^r a_{\epsilon}(r)dr=\int_{r_0}^r a(r)dr,
\]
so $\mu(P'_r)-\mu(P'_{r_0})=\int_{r_0}^r a(r)dr$, and thus $\mu(P'_r)$ is continuous for $r>0$.
Then $\frac{d\mu(P'_r)}{dr}=a(r)$
except at a finite number of points where $P'_r$ changes its combinatorial type
(and $a(r)$ may not be continuous or even bounded).
In (\ref{equation_schlafli_special_2}), multiplying by $\kappa=-1$,
we prove (\ref{equation_SDF_special_2}) and finish the proof.
\end{proof}

\section{Invariance properties}
\label{section_invariance_properties}

In the upper half-space model, 
if a polytope $P$ is isometric to $Q$ and both are in a finite region in $\mathbb{R}^n$,
then the isometry can be expressed as a finite composition of 
translations, similarities, orthogonal transformations and inversions.
In fact it need to involve up to only \emph{one} inversion,
so the polytopes in the middle steps of those basic types of isometries 
are all in a finite region in $\mathbb{R}^n$ too.
By the definition in (\ref{equation_volume_P_half_space_epsilon}),
$\mu(P)$ is clearly invariant under those basic types of isometries
except that it is not obvious for inversion. 
So to show that $\mu(P)$ is invariant under isometry we only need to show that
it is invariant under inversion.

As before,
let $B_r$ be a Euclidean ball centered at $O$ with radius $r>0$ where $r^2=x_0^2+\cdots+x_{n-1}^2$,
and $S^{n-1}_r$ be the $(n-1)$-sphere with radius $r$.

\begin{proposition}
\label{proposition_mu_u_invariant}
Assuming Theorem~\ref{theorem_volume_invariant} is true up to dimension $n-1\ge 1$.
Let $P$ be a polytope in $\mathbb{DH}^n$,
then $\mu(P)$ is invariant under inversion,
and thus $\mu(P)$ is invariant under isometry.
\end{proposition}

\begin{proof}
Without loss of generality, in the upper half-space model
let $g$ be an inversion that maps $P$ to $Q$ with $y_i=x_i/(x_0^2+\cdots+x_{n-1}^2)$.
As it is assumed that both $P$ and $Q$ are in a finite region in $\mathbb{R}^n$,
thus they also do not contain a neighborhood of the origin $O$.
Denote $P\cap B_r$ by $P'_r$. 
By Lemma~\ref{lemma_SDF_special_2} 
\begin{equation}
\label{equation_differential_volume_P}
\kappa\cdot\frac{d\mu(P'_r)}{dr}
=\frac{1}{n-1}\sum_{F} V_{n-2}(F)\,\frac{d\theta_F}{dr}.
\end{equation}
Similarly, we denote $Q\cap B_r$ by $Q'_r$, and for all $(n-2)$-faces $F'$ 
of $Q'_r$ on $S^{n-1}_r$, let $\varphi_{F'}$ be the dihedral angle at $F'$. Then
\begin{equation}
\label{equation_differential_volume_Q}
\kappa\cdot\frac{d\mu(Q'_r)}{dr}
=\frac{1}{n-1}\sum_{F'} V_{n-2}(F')\,\frac{d\varphi_{F'}}{dr}.
\end{equation}

When the sphere $S^{n-1}_r$ sweeps through $P$,
we have $S^{n-1}_{1/r}$ sweeping through $Q$ from the opposite direction,
so $P'_r$ is isometric to $Q\setminus Q'_{1/r}$.
For any $(n-2)$-face $F$ of $P'_r$ on $S^{n-1}_r$,
it is isometric to the corresponding $(n-2)$-face $F'$ of $Q'_{1/r}$ on $S^{n-1}_{1/r}$,
and $\theta_F=\pi-\varphi_{F'}$.
As it is assumed that Theorem~\ref{theorem_volume_invariant} is true for dimension $n-2$,
so $V_{n-2}(F)=V_{n-2}(F')$.
In (\ref{equation_differential_volume_Q}) replace $d\mu(Q'_r)$ with $d\mu(Q'_{1/r})$,
and add it to (\ref{equation_differential_volume_P}) and drop off $\kappa$, then
\[\frac{d\mu(P'_r)}{dr}+\frac{d\mu(Q'_{1/r})}{dr}=0.
\]
By Lemma~\ref{lemma_SDF_special_2} both $\mu(P'_r)$ and $\mu(Q'_{1/r})$
are continuous for $r>0$, so $\mu(P'_r)+\mu(Q'_{1/r})$ is a constant,
and for very small value of $r$ it is $\mu(Q)$,
while for very large value of $r$ it is $\mu(P)$.
So $\mu(P)=\mu(Q)$. Then $\mu(P)$ is invariant under inversion,
and thus $\mu(P)$ is invariant under isometry.
\end{proof}

Next we show that in the hemisphere model $\mu_h(P)$ is invariant under isometry.
Proposition~\ref{proposition_mu_h_invariant} is an important middle step 
as it is used to prove Proposition~\ref{proposition_mu_h_prime_invariant},
which will be applied many times in the induction step to prove Theorem~\ref{theorem_volume_invariant}.

\begin{proposition}
\label{proposition_mu_h_invariant}
Assuming Theorem~\ref{theorem_volume_invariant} is true up to dimension $n-1\ge 1$.
For a fixed $m$, let $P$ be a polytope in $\mathbb{DH}^n$ 
and be the intersection of at most $m$ half-spaces.
Then in the hemisphere model on $n$-sphere $S^n_r$,
(1) $\mu_h(P)$ is invariant under isometry,
and agrees with the value $\mu(g(P))$ for any isometry $g$ mapping 
to the upper half-space model, and 
(2) $\mu_{h,\epsilon}(P)$ is uniformly bounded.
\end{proposition}

\begin{proof}
In the hemisphere model, without loss of generality we use  the unit $n$-sphere $\mathbb{S}^n$
(see Remark~\ref{remark_hemisphere_model_r}).
Let $(t_0,t_1)=(-\frac{\sqrt{2}}{2},\frac{\sqrt{2}}{2})$,
and we use two planes $x_n=t_0$ and $x_n=t_1$
to cut $P$ into up to three pieces.
The ``middle'' piece satisfies $t_0\le x_n \le t_1$,
and both the ``side'' pieces
satisfy $t_0\le x_{n-1}\le t_1$ (not $x_n$).
(We intentionally choose the value to make them match,
but the value itself does not have any special meaning.)
By symmetry we only need to prove the statement for the middle piece,
so for convenience we assume $P$ satisfies $t_0\le x_n \le t_1$.

By (\ref{equation_volume_P_hemisphere_epsilon})
\[\mu_{h,\epsilon}(P)
=\int_{P\subset \mathbb{S}^n}\pm\frac{dx_1\cdots dx_n}{x_0(x_0-\epsilon i)^n}  
=\int_{P\subset \mathbb{S}^n}\pm\frac{x_0^n}{(x_0-\epsilon i)^n}\frac{dx_1\cdots dx_n}{x_0^{n+1}}.
\]
Let $g$ be a stereographic projection of the unit $n$-sphere
from $(0,\dots,0,-1)$ to the space $x_n=1$ (we use coordinates $y$'s for this space):
$g(x_0,\dots,x_n)=(y_0,\dots,y_{n-1})$.
See (\ref{equation_upper_half_space_coordinate}) for the formula of $g$.
As $g$ is an isometry, so $(g^{-1})^{\ast}$ maps 
the volume element $\pm\frac{dx_1\cdots dx_n}{x_0^{n+1}}$ 
in the hemisphere model (on $\mathbb{S}^n$)
into the volume element $\frac{dy_0\cdots dy_{n-1}}{y_0^n}$ 
in the upper half-space model with the appropriate orientation of each coordinate system,
thus
\begin{equation}
\label{equation_mu_h_P}
\mu_{h,\epsilon}(P)=\int_{g(P)\subset \mathbb{R}^n}\frac{x_0^n}{(x_0-\epsilon i)^n}\frac{dy_0\cdots dy_{n-1}}{y_0^n}
=\int_{g(P)\subset \mathbb{R}^n}\frac{dy_0\cdots dy_{n-1}}{(y_0-\frac{y_0}{x_0}\epsilon i)^n}.
\end{equation}
By (\ref{equation_upper_half_space_coordinate}), $\frac{y_0}{x_0}=\frac{2}{x_n+1}$.
Let $r^2=y_0^2+\cdots +y_{n-1}^2$.
As $x_n$ determines (and is determined by) $r$,
we denote $\frac{y_0}{x_0}$ by $c(r)$.
Further computation shows that $c(r)=\frac{4+r^2}{4}$, but we do not need it.

As $P$ does not contain a neighborhood of the points $x_n=\pm 1$, 
so $g(P)$ is bounded in $\mathbb{R}^n$ and does not contain a neighborhood of the origin of space $y$'s.
Let 
\[b_{\epsilon}(r)=\int_{g(P)\cap S^{n-1}_r}\frac{dA_r}{(y_0-c(r)\epsilon i)^n},
\]
where $dA_r$ is the Euclidean volume element on $(n-1)$-sphere $S^{n-1}_r$ (in space $y$'s).
In (\ref{equation_mu_h_P}) using the polar coordinates, then
\[\mu_{h,\epsilon}(P)=\int_{r_0}^{r_1}b_{\epsilon}(r)dr,
\]
where $0<r_0<r_1$ is determined by $t_0$ and $t_1$.
Also let 
\[b(r)=\lim_{\epsilon\to 0^+} b_{\epsilon}(r).
\]
Notice that if we use the notation $a_{\epsilon}(r)$ from (\ref{equation_derivative_u_epsilon_r}),
and ignore the differences between $P$ and $g(P)$ for a moment,
then $b_{\epsilon}(r)=a_{c(r)\epsilon}(r)$.
In (\ref{equation_schlafli_special_epsilon_2_bound}), 
on $[r_0,r_1]$ $a_{\epsilon}(r)$ is bounded by 
an integrable function that does not depend on the value of $\epsilon$,
therefore $b_{\epsilon}(r)$ is also bounded by the same integrable function
with $c(r)$ playing no role in the argument.
Then Lemma~\ref{lemma_a_epsilon_limit_convergent} also applies to $\{b_{\epsilon}(r)\}$
(replacing $\{a_{\epsilon}(r)\}$),
thus Lebesgue dominated convergence theorem applies to $\{b_{\epsilon}(r)\}$
and $b(r)$ is integrable on $[r_0,r_1]$, and
\[\mu_h(P)
=\lim_{\epsilon\to 0^+} \mu_{h,\epsilon}(P)         
=\lim_{\epsilon\to 0^+} \int_{r_0}^{r_1} b_{\epsilon}(r)dr =\int_{r_0}^{r_1} b(r) dr.
\]
Also for fixed $m$, $r_0$ and $r_1$,
by Lemma~\ref{lemma_a_epsilon_limit_convergent}
we have $\int_{r_0}^{r_1}b_{\epsilon}(r)dr$ uniformly bounded.
As $r_0$ and $r_1$ are determined by fixed constants 
$(t_0,t_1)=(-\frac{\sqrt{2}}{2},\frac{\sqrt{2}}{2})$,
so for a fixed $m$, $\mu_{h,\epsilon}(P)$ is uniformly bounded.
This proves (2).

As the $c(r)$ in $b_{\epsilon}(r)$ can be dropped without affecting the calculation of $b(r)$,
so by applying Lebesgue dominated convergence theorem again, we have

\begin{align*}
\int_{r_0}^{r_1} b(r) dr
&=\lim_{\epsilon\to 0^+} \int_{r_0}^{r_1} \left(\int_{g(P)\cap S^{n-1}_r}\frac{dA_r}{(y_0-\epsilon i)^n}\right) dr   \\
&=\lim_{\epsilon\to 0^+} \int_{g(P)\subset \mathbb{R}^n}\frac{dy_0\cdots dy_{n-1}}{(y_0-\epsilon i)^n}  
=\mu(g(P)),
\end{align*}
where the second step uses polar coordinates to $g(P)$ in $\mathbb{R}^n$
and the last step is by Lemma~\ref{lemma_SDF_special}.
So $\mu_h(P)=\mu(g(P))$.
So by Proposition~\ref{proposition_mu_u_invariant},
$\mu_h(P)$ is also invariant under isometry.
This proves (1) and finishes the proof.
\end{proof}

In the hemisphere model on $n$-sphere $S^n_r$, let
\begin{equation}
\label{equation_P_hemisphere_n_plus_1}
g_{h,\epsilon}(P)
=\int_{P\subset S^n_r} \pm\frac{rdx_1\cdots dx_n}{x_0(x_0-\epsilon i)^{n+1}},
\end{equation}
with the positive sign for $x_0>0$ and the minus sign for $x_0<0$.
By (\ref{equation_volume_P_hemisphere_r_epsilon})
and (\ref{equation_volume_P_hemisphere_r_epsilon_variant}), 
we have
\begin{equation}
\label{equation_mu_h_epsilon_P_difference}
\mu'_{h,\epsilon}(P)-\mu_{h,\epsilon}(P)
=\epsilon g_{h,\epsilon}(P) i.
\end{equation}

Finally, as the last part of the induction step to prove Theorem~\ref{theorem_volume_invariant},
we show that $\mu'_h(P)$ is invariant under isometry.

\begin{proposition}
\label{proposition_mu_h_prime_invariant}
Assuming Theorem~\ref{theorem_volume_invariant} is true up to dimension $n-1\ge 1$.
For a fixed $m$, let $P$ be a polytope in $\mathbb{DH}^n$ 
and be the intersection of at most $m$ half-spaces.
Then in the hemisphere model on $n$-sphere $S^n_r$,
(1) $\mu'_h(P)$ is  invariant under isometry and $\mu'_h(P)=\mu_h(P)$,
and 
(2) $\mu'_{h,\epsilon}(P)$ is uniformly bounded.
\end{proposition}

\begin{proof}
In (\ref{equation_P_hemisphere_n_plus_1}), by using polar coordinates, we have
\[g_{h,\epsilon}(P)
=\int_{P\subset S^n_r} \frac{dA_r}{(x_0-\epsilon i)^{n+1}},
\]
where $dA_r$ is the Euclidean volume element on $n$-sphere $S^n_r$.
Notice that $g_{h,\epsilon}(P)$ has the same form as $a_{\epsilon}(r)$ in (\ref{equation_derivative_u_epsilon_r}),
except that in (\ref{equation_derivative_u_epsilon_r}) $P$ is in a different context
and the integration is on an $(n-1)$-dimensional region instead.
In fact, the formula of $a_{\epsilon}(r)$ in 
(\ref{equation_schlafli_special_epsilon_2_formula}) 
does not rely on any induction assumption
and can be freely applied to higher dimensions to compute $g_{h,\epsilon}(P)$ as well.
By using (\ref{equation_schlafli_special_epsilon_2_formula}) 
and working carefully with those minor notational differences, we have
\begin{equation}
\label{equation_P_hemisphere_n_plus_1_formula}
g_{h,\epsilon}(P)
=-\frac{1}{n}\sum_{E}\pm \frac{(r^2-r_{\sss E}^2)^{1/2}}{r\cdot r_{\sss E}} \mu'_{h,\epsilon}(E)
+\frac{\epsilon i}{r^2}\mu'_{h,\epsilon}(P),
\end{equation}
where the sum is taken over all $(n-1)$-faces $E$ of $P$, and $r_{\sss E}$ is the radius of $E$.
The ``$\pm$'' is explained in Lemma~\ref{lemma_theta_derivative_r}, 
but this sign is not important for our purpose 
as we are only concerned with its magnitude.

In (\ref{equation_P_hemisphere_n_plus_1_formula}), multiplying $\epsilon i$ on both sides
and then substituting the left side with $\mu'_{h,\epsilon}(P)-\mu_{h,\epsilon}(P)$
(by (\ref{equation_mu_h_epsilon_P_difference})), we have
\[\mu'_{h,\epsilon}(P)-\mu_{h,\epsilon}(P)
=-\frac{1}{n}\sum_{E}\pm \frac{(r^2-r_{\sss E}^2)^{1/2}\epsilon i}{r\cdot r_{\sss E}} \mu'_{h,\epsilon}(E)
-\frac{\epsilon^2}{r^2}\mu'_{h,\epsilon}(P).
\]
Rearrange the terms, then
\[(1+\frac{\epsilon^2}{r^2})\mu'_{h,\epsilon}(P)
=-\frac{1}{n}\sum_{E}\pm \frac{(r^2-r_{\sss E}^2)^{1/2} i}{r} \cdot \frac{\epsilon}{r_{\sss E}} \mu'_{h,\epsilon}(E)
+\mu_{h,\epsilon}(P),
\]
thus
\begin{equation}
\label{equation_P_hemisphere_formula}
\mu'_{h,\epsilon}(P)
=-\frac{r^2}{n(r^2+\epsilon^2)}\sum_{E}\pm \frac{(r^2-r_{\sss E}^2)^{1/2} i}{r} \cdot \frac{\epsilon}{r_{\sss E}} \mu'_{h,\epsilon}(E)
+\frac{r^2}{r^2+\epsilon^2}\mu_{h,\epsilon}(P).
\end{equation}

For a fixed $m$, by Lemma~\ref{lemma_f_u_epsilon_limit}
$\frac{\epsilon}{r_{\sss E}}\mu'_{h,\epsilon}(E)$ is uniformly bounded,
and by Proposition~\ref{proposition_mu_h_invariant}
$\mu_{h,\epsilon}(P)$ is uniformly bounded.
As $\frac{(r^2-r_{\sss E}^2)^{1/2}}{r}\le 1$ and $\frac{r^2}{r^2+\epsilon^2}\le 1$,
so the right side of (\ref{equation_P_hemisphere_formula}) is uniformly bounded.
Thus $\mu'_{h,\epsilon}(P)$ is also uniformly bounded
and this proves (2).

By the uniform boundedness of $\mu'_{h,\epsilon}(E)$ by induction,
$\lim_{\epsilon\to 0^+} \epsilon\mu'_{h,\epsilon}(E)=0$.
Then in (\ref{equation_P_hemisphere_formula}) as $\epsilon\to 0^+$ we have
\[\mu'_h(P)=\lim_{\epsilon\to 0^+} \mu'_{h,\epsilon}(P)
=\lim_{\epsilon\to 0^+} \frac{r^2}{r^2+\epsilon^2}\mu_{h,\epsilon}(P)=\mu_h(P),
\]
where the last step is by Proposition~\ref{proposition_mu_h_invariant}.
This proves (1) and finishes the proof.
\end{proof}

\section{Proof of Theorem~\ref{theorem_volume_invariant}}
\label{section_volume_invariant_proof}

The proof of Theorem~\ref{theorem_volume_invariant} is simply a summarization of early results.

\begin{proof}[Proof of Theorem~\ref{theorem_volume_invariant}]
For dimension 0 and 1, Theorem~\ref{theorem_volume_invariant} 
is verified in Lemma~\ref{lemma_volume_invariant_0_1}.
Now assuming it is true up to dimension $n-1\ge 1$, and for a fixed $m$,
let $P$ be an $n$-dimensional polytope in $\mathbb{DH}^n$
and be the intersection of at most $m$ half-spaces.
By Lemma~\ref{lemma_SDF_special}, 
$\mu(P)$ exists and $\mu_{\epsilon}(P)$ is uniformly bounded.
By Proposition~\ref{proposition_mu_u_invariant}, $\mu(P)$ is invariant under isometry.
By Proposition~\ref{proposition_mu_h_invariant}, $\mu_h(P)$ is invariant under isometry
and agrees with the value $\mu(g(P))$ for any isometry $g$ mapping to
the upper half-space model,
and $\mu_{h,\epsilon}(P)$ is uniformly bounded.
By Proposition~\ref{proposition_mu_h_prime_invariant}, 
$\mu'_h(P)$ is invariant under isometry and $\mu'_h(P)=\mu_h(P)$,
and $\mu'_{h,\epsilon}(P)$ is uniformly bounded.
So we prove Theorem~\ref{theorem_volume_invariant} for dimension $n$ by induction,
and finish the proof.
\end{proof}

\section{Proof of Theorem~\ref{theorem_schlafli}}
\label{section_schlafli_double_proof}

With Theorem~\ref{theorem_volume_invariant} proved,
in this section we prove Theorem~\ref{theorem_schlafli}. 
We first have the following analogue of a classical formula for polygons in dimension two.

\begin{lemma}
\label{lemma_volume_polygon_mu}
For $n=2$, let $P$ be a polytope in $\mathbb{DH}^2$.
If $P_{+}$ in $\mathbb{H}^2$ has $m$ sides with dihedral angles $\theta_i$ between consecutive sides
(the sides in $\mathbb{H}^2$ can always be arranged in a circular order,
and denote $\theta_i=0$ if two consecutive sides do not intersect),
then
\begin{equation}
\label{equation_volume_polygon_mu}
V_2(P)=2(m-2)\pi-2\sum_i\theta_i.
\end{equation}
\end{lemma}

\begin{proof}
We will use the upper half-space model.
By Lemma~\ref{lemma_SDF_special} we have $V_2(P)=\mu(P)=\int f(t)dt$ 
where by (\ref{equation_schlafli_special}) $f(t)=-2\sum_{F}\frac{d\theta_F}{dt}$
(the coefficient 2 is because each 0-face $F$ has two points and $V_0(F)=2$).
By integrating $f(t)$ over $t$, with the details skipped,
we then verify (\ref{equation_volume_polygon_mu}).
\end{proof}

In Theorem~\ref{theorem_schlafli}, 
the assumption that $P$ does not have any ideal vertices
will greatly simplify the proof of the \SDF{} for $\mathbb{DH}^n$
without affecting other main results.

\begin{theorem}
\emph{(Theorem~\ref{theorem_schlafli}, Schl\"{a}fli differential formula for $\mathbb{DH}^n$)}
For $n\ge 2$, let $P$ be a polytope in $\mathbb{DH}^n$ that does not contain any ideal vertices, 
and for each $(n-2)$-dimensional face $F$, let $\theta_F$ be the dihedral angle at $F$.
Then for $\kappa=-1$,
\begin{equation}
\label{equation_schlafli_double_hyperbolic}
\kappa\cdot dV_n(P)=\frac{1}{n-1}\sum_{F}V_{n-2}(F)\,d\theta_F,
\end{equation}
where the sum is taken over all $(n-2)$-faces $F$ of $P$.
For $n-2=0$, $V_0(F)$ is the number of points in $F$.
\end{theorem}

\begin{proof}
For $n=2$, by Lemma~\ref{lemma_volume_polygon_mu} we have
$V_2(P)=2(m-2)\pi-2\sum_i\theta_i$,
so it automatically satisfies (\ref{equation_schlafli_double_hyperbolic}).
We assume $n\ge 3$ in the following.

For an arbitrary polytope $P$ in $\mathbb{DH}^n$,
first we want to cut $P$ into some subdivisions, 
such that each subdivision is a polytope and all the $(n-1)$-dimensional faces meet transversally.
In the Klein model for $\mathbb{DH}^n$, which is a double covering of a disk,
if we only consider the upper portion $P_{+}$ of $P$,
then $P_{+}$ lies on a real projective space $\mathbb{RP}^n$.
We can find a bounded region $P^{\prime}$ in $\mathbb{RP}^n$ with flat facets only
(like a Euclidean polytope),
such that each flat facet of $P_{+}$ is also part of a flat facet of $P^{\prime}$.
With $P^{\prime}$ we can have a standard triangulation with flat planes,
and when restricting this triangulation to $P_{+}$ and performing the same cut
on the lower portion $P_{-}$ of $P$, we cut $P$ into some pieces of polytopes in $\mathbb{DH}^n$.
As $P$ does not contain ideal vertices,
the cut can be general enough such that for each subdivision there is still no ideal vertices.
If we vary $P$ smoothly in a small neighborhood, then we can vary each subdivision 
smoothly as well. As both sides of (\ref{equation_schlafli_double_hyperbolic}) are additive 
with respect to those subdivisions, thus proving (\ref{equation_schlafli_double_hyperbolic})
for a general $P$ can be deduced to proving for each subdivision of $P$.
For any subdivision of $P$, as all its $(n-1)$-dimensional faces meet transversally,
so any small variation can be obtained by varying the $(n-1)$-faces independently.

Without loss of generality, we still denote the subdivision by $P$.
It suffices to consider just varying any one of the $(n-1)$-faces $E$.
By Lemma~\ref{lemma_SDF_special}, it shows that (\ref{equation_schlafli_double_hyperbolic})
is true for a special one-parameter family of polyhedra.
Next we extend it to an $n$-parameter family where $E$ can move freely.

We first use the hyperboloid model.
Let $H$ be the $(n-1)$-dimensional $\mathbb{DH}^{n-1}$ that contains $E$,
and $e$ be the inward unit normal to $P$ along the face $E$ in the upper sheet.
We have $e\cdot e=1$ and any small movement of $e$ corresponds to a small movement of $E$.
As both sides of (\ref{equation_schlafli_double_hyperbolic}) are additive 
with respect to subdivisions of $P$, without loss of generality we assume $E$ is not 
a full $\mathbb{DH}^{n-1}$.
As $n\ge 3$, there is a point $p$ in $\partial\mathbb{H}^n\cap H$ (a $\partial\mathbb{H}^{n-1}$)
but outside of $E$, thus $p$ is also outside of $P$.

Fix any finite $p^{\prime}$ on the light cone to represent $p$.
Since $p$ is on $H$, so $p^{\prime}\cdot e=0$.
As on light cone we have $p^{\prime}\cdot p^{\prime}=0$, 
so $(e+sp^{\prime})\cdot (e+sp^{\prime})=1$,
hence $e$ may vary
along the line of $e+sp^{\prime}$ while still being a unit vector.
Now switching to the upper half-space model
by mapping $p$ (under an isometry) to the point at infinity $x_0=\infty$
(so $P$ is mapped to a finite region in $\mathbb{R}^n$).
Then varying $e$ along the line of $e+sp^{\prime}$ corresponds to a moving plane 
of the form $x_{n-1}=t$ sweeping through the Euclidean space parallelly.
By Lemma~\ref{lemma_SDF_special},
this special case satisfies (\ref{equation_schlafli_double_hyperbolic}).

Now switch back to the hyperboloid model.
As $E$ intersects other faces transversally,
$dV_n(P)$ is linear to the change of $e$.
To show that (\ref{equation_schlafli_double_hyperbolic})
is true when $e$ varies along any direction in a small neighborhood, 
it suffices to find $n$ linearly independent vectors that satisfy
(\ref{equation_schlafli_double_hyperbolic}) as $p$ does.
As $n\ge 3$, this can be achieved by selecting $n$ such linearly independent rays 
$p_1,\dots,p_n$ in a small neighborhood of $p$ on $\partial\mathbb{H}^n\cap H$
(a $\partial\mathbb{H}^{n-1}$). This completes the proof.
\end{proof}

\section{Proofs of Theorem~\ref{theorem_total_volume}, \ref{theorem_volume_finite}, 
and \ref{theorem_volume_real_imaginary}}
\label{section_total_volume_second}

As an elementary application of the \SDF{} for $\mathbb{DH}^n$ (Theorem~\ref{theorem_schlafli}),
we prove Theorem~\ref{theorem_total_volume},
with a proof similar to \cite[Example 2]{Milnor:Schlafli}
about standard unit spheres.

\begin{theorem}
\emph{(Theorem~\ref{theorem_total_volume})}
The total volume of $\mathbb{DH}^n$ is 
\begin{equation}
V_n(\mathbb{DH}^n)=i^n V_n(\mathbb{S}^n)
\end{equation}
for both even and odd dimensions,
where $V_n(\mathbb{S}^n)$ is the $n$-dimensional volume of the standard 
unit $n$-sphere $\mathbb{S}^n$.
\end{theorem}

\begin{proof}
In the hemisphere model on the unit $n$-sphere $\mathbb{S}^n$:
$x_0^2+\cdots+x_n^2=1$, where $\partial\mathbb{H}^n$ is on $x_0=0$, 
define an $n$-dimensional \emph{lune} $L^n_{\theta}$
to be the portion of $\mathbb{DH}^n$ such that
the last two coordinates can be expressed as
\[x_{n-1}=r\cos\varphi, \quad x_n=r\sin\varphi, 
\quad\text{with}\quad r\geq 0\quad\text{and}\quad
0\leq\varphi\leq\theta.
\]
Then $L^n_{\theta}$ has just two $(n-1)$-dimensional faces
of $\varphi=0$ and $\varphi=\theta$.
Their intersection is an $(n-2)$-dimensional $\mathbb{DH}^{n-2}$
with $x_{n-1}=x_n=0$, and with a dihedral angle $\theta$ between the two faces.
So by Theorem~\ref{theorem_schlafli}, we have
\[\kappa\cdot dV_n(L^n_{\theta})
=\frac{1}{n-1}V_{n-2}(\mathbb{DH}^{n-2})d\theta.
\]
As $V_n(L^n_{\theta})=0$ for $\theta=0$, so integrating the above
one gets
\[\kappa\cdot V_n(L^n_{\theta})
=\frac{1}{n-1}V_{n-2}(\mathbb{DH}^{n-2})\theta.
\]
Taking $\theta=2\pi$, then $L^n_{\theta}$ is the full $\mathbb{DH}^n$,
and therefore
\[\kappa\cdot V_n(\mathbb{DH}^n)
=\frac{2\pi}{n-1}V_{n-2}(\mathbb{DH}^{n-2}).
\]
Notice that with the exception of the extra coefficient $\kappa$,
it is exactly the same recursive formula for $V_n(\mathbb{S}^n)$.
As $\kappa=-1=i^2$, as well as
$V_0(\mathbb{DH}^0)=2=V_0(\mathbb{S}^0)$ 
and $V_1(\mathbb{DH}^1)=2\pi i=i V_1(\mathbb{S}^1)$
of (\ref{equation_total_volume_dh1}),
we immediately obtain
\[V_n(\mathbb{DH}^n)=i^n V_n(\mathbb{S}^n).
\]
This finishes the proof.
\end{proof}

Let $\mathcal{H}$ be the algebra over $\mathbb{DH}^n$
generated by half-spaces in $\mathbb{DH}^n$,
for $n\ge 2$ we show that $V_n(P)$ is not countably additive on $\mathcal{H}$.

\begin{example}
\label{example_not_countably_additive_P}
For $n\ge 2$, in the upper half-space model of $\mathbb{DH}^n$,
let $B_i$ be a closed ball centered at $O$ with radius $1/i$.
Let $P_i=B_i\setminus B_{i+1}$,  then $B_1\setminus\{O\}=\bigcup_{i=1}^{\infty}P_i$.
Notice that $\{O\}$ is the intersection of two closed half-spaces tangent at $O$,
so $\{O\} \in\mathcal{H}$ and hence $B_1\setminus\{O\} \in\mathcal{H}$.
As a ball is a half-space and always has non-zero volume $\frac{1}{2}V_n(\mathbb{DH}^n)$, 
so $V_n(\bigcup_{i=1}^{\infty}P_i)=V_n(B_1\setminus\{O\})\ne 0$.
We also have $V_n(P_i)=V_n(B_i)-V_n(B_{i+1})=0$, 
so $V_n(\bigcup_{i=1}^{\infty}P_i)\ne\sum_{i=1}^{\infty}V_n(P_i)$.
Thus $V_n(P)$ is not countably additive on $\mathcal{H}$.
\end{example}

\begin{theorem}
\emph{(Theorem~\ref{theorem_volume_finite}, Uniform boundedness of $V_n(P)$ for a fixed $m$)}
Let $P$ be a polytope in $\mathbb{DH}^n$
and be the intersection of at most $m$ half-spaces in $\mathbb{DH}^n$, then 
\begin{equation}
\label{equation_volume_finite_repeat}
|V_n(P)|\le\frac{m!}{2^{m-1}}V_n(\mathbb{S}^n).
\end{equation}
\end{theorem}

The bound for $V_n(P)$ in (\ref{equation_volume_finite_repeat}) is very loose, 
but it provides an explicit bound
and can be verified by running a rather simple induction on both $m$ and $n$.
On the other hand,  there is no fixed bound for $V_n(P)$ for all $m$.
For example, when $P$ is obtained by ``cutting'' $m$ non-intersecting half-spaces
from $\mathbb{DH}^n$ for $m\ge 2$, then $|V_n(P)|=\frac{m-2}{2}\,V_n(\mathbb{S}^n$).

\begin{proof}
First, for all $m\ge 0$ we have $\frac{m!}{2^{m-1}}\ge 1$.
For $n=0$, $P$ is $\mathbb{DH}^0$, so (\ref{equation_volume_finite_repeat}) is true.
For $n=1$, as $V_1(P)$ only takes values of $2\pi i$, $\pi i$, or $0$,
so $|V_1(P)|\le V_1(\mathbb{S}^1)$ and (\ref{equation_volume_finite_repeat}) is true as well.
We assume $n\ge 2$ in the following.
For $m\le 1$, $P$ is either the full $\mathbb{DH}^n$ or a half-space in $\mathbb{DH}^n$,
so by Theorem~\ref{theorem_total_volume} we have
$|V_n(P)|\le V_n(\mathbb{S}^n)$ and thus (\ref{equation_volume_finite_repeat}) is also true.
Now we run induction on $m$ for $m\ge 2$.

We use the upper half-space model for $\mathbb{DH}^n$,
and without loss of generality, we assume $P$ is in a finite region in $\mathbb{R}^n$.
Denote by $H_i$ the face of each half-space, which is also a $\mathbb{DH}^{n-1}$.
Let $P_t$ be the intersection of $P$ and $x_{n-1}\le t$,
$E_i$ be the $(n-1)$-face of $P$ on $H_i$,
and $F_i$ be the intersection of $E_i$ and $x_{n-1}=t$
(it is ok if $F_i$ is an empty face).
As there are at most $m$ such $(n-1)$-faces of $P$,
by Lemma~\ref{lemma_SDF_special} 
(and replacing $\mu(P_t)$ with $V_n(P_t)$) we have
\begin{equation}
\label{equation_schlafli_special_repeat}
\kappa\cdot \frac{dV_n(P_t)}{dt}
=\frac{1}{n-1}\sum_{i\le m} V_{n-2}(F_i)\,\frac{d\theta_{F_i}}{dt}.
\end{equation}

For a given non-empty $(n-2)$-face $F_i$ of $P_t$ on $x_{n-1}=t$,
$F_i$ lies on a $\mathbb{DH}^{n-2}$ 
that is the intersection of $H_i$ and $x_{n-1}=t$.
In this $\mathbb{DH}^{n-2}$,
$F_i$ is the intersection of \emph{at most} $m-1$ half-spaces 
(with the possibility of $F_i$ being the full $\mathbb{DH}^{n-2}$),
contributed by other $H_j$ ($j\ne i$, but not $H_i$) intersecting with this $\mathbb{DH}^{n-2}$.
By induction we have 
\[|V_{n-2}(F_i)|\le \frac{(m-1)!}{2^{m-2}}\,V_{n-2}(\mathbb{S}^{n-2}),
\]
so by (\ref{equation_schlafli_special_repeat})
\begin{equation}
\label{equation_f_inequality}
\left| \frac{dV(P_t)}{dt} \right|
\le \frac{1}{n-1}\cdot\frac{(m-1)!}{2^{m-2}}\,V_{n-2}(\mathbb{S}^{n-2})\sum_{i\le m} \left|\frac{d\theta_{F_i}}{dt}\right|.
\end{equation}

When $F_i$ is not an empty face,
we also have $\theta_{F_i}$ as a \emph{monotonic} function over $t$
that takes values in an interval between $0$ and $\pi$. 
By this monotonicity of $\theta_{F_i}$, no matter increasing or decreasing, 
integrating in this range only we have $\int |\frac{d\theta_{F_i}}{dt}|dt\le\pi$.
Thus by integrating (\ref{equation_f_inequality})

\begin{align*}
\int \left| \frac{dV(P_t)}{dt} \right| dt
&\le\frac{1}{n-1}\cdot\frac{(m-1)!}{2^{m-2}}\,V_{n-2}(\mathbb{S}^{n-2})\cdot m\pi       \\
&=\frac{2\pi}{n-1}\cdot\frac{m!}{2^{m-1}}\,V_{n-2}(\mathbb{S}^{n-2})           
=\frac{m!}{2^{m-1}}\,V_n(\mathbb{S}^n),
\end{align*}
where the last step uses a recursive formula 
$V_n(\mathbb{S}^n)=\frac{2\pi}{n-1}V_{n-2}(\mathbb{S}^{n-2})$.
So
\[|V_n(P)|
\le\int \left| \frac{dV(P_t)}{dt} \right| dt
\le \frac{m!}{2^{m-1}}\,V_n(\mathbb{S}^n).
\]
This completes the proof.
\end{proof}

We also immediately prove Theorem~\ref{theorem_volume_real_imaginary}.

\begin{theorem}
\emph{(Theorem~\ref{theorem_volume_real_imaginary})}
Let $\mathcal{H}$ be the algebra over $\mathbb{DH}^n$ generated by half-spaces in $\mathbb{DH}^n$, and $P\in\mathcal{H}$.
Then $V_n(P)$ is real for $n$ even,
and $V_n(P)$ is imaginary for $n$ odd and is completely determined by $P\cap \partial\mathbb{H}^n$.
\end{theorem}

\begin{proof}
In the upper half-space model,
without loss of generality we assume that $P$ is in a finite region in $\mathbb{R}^n$.
By (\ref{equation_mu_epsilon_P_plus}) we have
\[\mu_{\epsilon}(P)=\mu_{\epsilon}(P_{+})+\mu_{\epsilon}(P_{-})
=\mu_{\epsilon}(P_{+})+(-1)^n \mu_{-\epsilon}(P_{+}).
\]
Taking the pointwise sum of ($\mu_{\epsilon}+(-1)^n \mu_{-\epsilon})$ on $P_{+}$, 
observe that it is real for $n$ even and imaginary for $n$ odd.
As $V_n(P)=\mu(P)=\lim_{\epsilon\to 0^+} \mu_{\epsilon}(P)$,
so $V_n(P)$ is also real for $n$ even and imaginary for $n$ odd.

For $n=1$, if $P\cap\partial\mathbb{H}^1$ contains 0, 1, or 2 points,
$V_1(P)$ takes value of 0, $\pi i$, or $2\pi i$ respectively.
So $V_1(P)$ is determined by the number of points in  $P\cap\partial\mathbb{H}^1$.

For $n$ odd and $n\ge 3$, assume $P, P^{\prime}\in\mathcal{H}$ and
$P\cap\partial\mathbb{H}^n=P^{\prime}\cap\partial\mathbb{H}^n$.
We first show that $V_n(P\setminus P^{\prime})=0$.
It is convenient to use the Klein model for visualization,
and denote the upper portion of $P\setminus P^{\prime}$ by $(P\setminus P^{\prime})_{+}$.
Using flat planes $(P\setminus P^{\prime})_{+}$
can be cut into the union of hyperbolic simplices 
(with ideal vertices allowed, see also Remark~\ref{remark_union_polytopes}),
so $(P\setminus P^{\prime})_{+}$ has finite volume.
As $n$ is odd, so $V_n((P\setminus P^{\prime})_{-})=-V_n((P\setminus P^{\prime})_{+})$,
hence $V_n(P\setminus P^{\prime})=0$. 
Therefore
\[V_n(P)=V_n(P\setminus P^{\prime}) + V_n(P\cap P^{\prime})=V_n(P\cap P^{\prime}).
\]
By symmetry we also have $V_n(P^{\prime})=V_n(P\cap P^{\prime})$,
therefore $V_n(P^{\prime})=V_n(P)$.
Thus for $n$ odd, $V_n(P)$ is completely determined by $P\cap\partial\mathbb{H}^n$.
\end{proof}

Theorem~\ref{theorem_volume_real_imaginary} is crucial for exploring new 
geometric properties of $\partial\mathbb{H}^n$ on top of the standard conformal structure on a sphere.
We will discuss this in more detail next.

\section{A volume on $\partial\mathbb{H}^{2m+1}$}
\label{section_volume_boundary}

As an important application of Theorem~\ref{theorem_volume_real_imaginary},
for $n$ odd, we show that the volume on $\mathbb{DH}^n$ induces an intrinsic volume
on $\partial\mathbb{H}^n$ on regions generated by half-spaces in $\partial\mathbb{H}^n$,
with the induced volume invariant under M\"{o}bius transformations.
But this volume on $\partial\mathbb{H}^n$ is not induced by any volume form 
on $\partial\mathbb{H}^n$ as a differentiable manifold.

Here we clarify some notions that apply to all $n\ge 1$, both even and odd.
For a half-space in $\mathbb{DH}^n$, its restriction to $\partial\mathbb{H}^n$ is 
called a \emph{half-space} in $\partial\mathbb{H}^n$.
If using $\mathbb{R}^{n-1}$ (plus a $\infty$) as a model for $\partial\mathbb{H}^n$,
then a half-space is either the insider or the outside of a ball, or a Euclidean half-space;
if using $\mathbb{S}^{n-1}$ as a model for $\partial\mathbb{H}^n$,
then any hyperplane crossing $\mathbb{S}^{n-1}$ 
cuts $\partial\mathbb{H}^n$ into two half-spaces.
Let $\mathcal{F}$ (resp. $\mathcal{H}$) be the algebra over $\partial\mathbb{H}^n$
(resp. $\mathbb{DH}^n$) generated 
by the half-spaces in $\partial\mathbb{H}^n$ (resp. $\mathbb{DH}^n$).
A \emph{polytope} in $\partial\mathbb{H}^n$ is 
a finite intersection of closed half-spaces in $\partial\mathbb{H}^n$,
which can also be viewed as a restriction of a polytope $P$ in $\mathbb{DH}^n$ to $\partial\mathbb{H}^n$
(but the choice of $P$ may not be unique, see Remark~\ref{remark_missing_face}).


For $n=2m+1$ and $P\in\mathcal{H}$, let $G=P\cap\partial\mathbb{H}^{2m+1}$.
We define a real valued \emph{volume} of $G$ by
\begin{equation}
\label{equation_volume_boundary}
V_{\infty,2m}(G):=c_{2m}\cdot V_{2m+1}(P),
\quad\text{where}\quad
c_{2m}=\frac{V_{2m}(\mathbb{S}^{2m})}{i^{2m+1}V_{2m+1}(\mathbb{S}^{2m+1})}.
\end{equation}
The factors $c_{2m}$ are chosen to normalize $V_{2m+1}(P)$ in a way such that 
(1) for $m=0$, $V_{\infty,0}(G)$ is the number of points in $G$, and
(2) for a polytope $G$, $V_{\infty,2m}(G)$ satisfies a \SDF{}  for $\partial\mathbb{H}^{2m+1}$ 
with the appropriate coefficients (Theorem~\ref{theorem_schlafli_boundary_infinity_new}).
Setting $P=\mathbb{DH}^{2m+1}$, 
by Theorem~\ref{theorem_total_volume} and (\ref{equation_volume_boundary}), we have
\begin{equation}
\label{equation_volume_boundary_total}
V_{\infty,2m}(\partial\mathbb{H}^{2m+1})=V_{2m}(\mathbb{S}^{2m}).
\end{equation}
It is worth noting that the choice of the set of factors $c_{2m}$ is not unique,
e.g., we can just as well choose a different set of factors such that
$V_{\infty,2m}(\partial\mathbb{H}^{2m+1})=(-1)^m V_{2m}(\mathbb{S}^{2m})$ instead,
then the \SDF{} for $\partial\mathbb{H}^{2m+1}$
in Theorem~\ref{theorem_schlafli_boundary_infinity_new}
should add a minus sign.

\begin{theorem}
\label{theorem_invariant_mobius}
For $m\ge 0$, $V_{\infty,2m}(G)$ is well defined on $\mathcal{F}$,
and is invariant under M\"{o}bius transformations of $\partial\mathbb{H}^{2m+1}$.
\end{theorem}

\begin{proof}
For $G\in\mathcal{F}$, let $P, P^{\prime}\in\mathcal{H}$ and satisfy
$P\cap\partial\mathbb{H}^{2m+1}=P^{\prime}\cap\partial\mathbb{H}^{2m+1}=G$.
By Theorem~\ref{theorem_volume_real_imaginary},
we have $V_{2m+1}(P)=V_{2m+1}(P^{\prime})$,
so $V_{\infty,2m}(G)$ is well defined and independent of the choice of $P\in\mathcal{H}$.
By Theorem~\ref{theorem_volume_invariant}, 
$V_{2m+1}(P)$ is invariant under isometry of $\mathbb{DH}^{2m+1}$,
therefore $V_{\infty,2m}(G)$ is invariant under M\"{o}bius transformations of $\partial\mathbb{H}^{2m+1}$.
\end{proof}

\begin{remark}
We remark that $V_{\infty,2m}(G)$ is not countably additive on $\mathcal{F}$.
This can be shown by using the sets in $\mathbb{DH}^{2m+1}$
from Example~\ref{example_not_countably_additive_P},
and then restricting them to $\partial\mathbb{H}^{2m+1}$.
\end{remark}

\begin{remark}
\label{remark_boundary_ball_volume}
If we use $\mathbb{R}^{2m}$ (plus a $\infty$) as a model for $\partial\mathbb{H}^{2m+1}$,
any ball in $\mathbb{R}^{2m}$ is a half-space in $\partial\mathbb{H}^{2m+1}$,
then by Theorem~\ref{theorem_invariant_mobius}
it has a fixed non-zero volume in $\partial\mathbb{H}^{2m+1}$.
As a ball can be arbitrarily ``small'' in $\mathbb{R}^{2m}$,
thus it implies that the volume $V_{\infty,2m}(G)$ is not induced by any volume form 
on $\mathbb{R}^{2m}$ as a differentiable manifold.
\end{remark}

Liouville's theorem states that all conformal mappings on a domain 
of $\mathbb{R}^n$ and $\mathbb{S}^n$ for $n\ge 3$ are restrictions of M\"{o}bius transformations. 
So the volume $V_{\infty,2m}(G)$ on $\partial\mathbb{H}^{2m+1}$ 
for $m>1$ is also invariant under conformal mappings.
To our knowledge $V_{\infty,2m}(G)$ is a new invariant we discovered on $\partial\mathbb{H}^{2m+1}$.
We have the following question.

\begin{question}
Can the volume $V_{\infty,2m}(G)$ be defined in $\partial\mathbb{H}^{2m+1}$
for a larger class of regions than $\mathcal{F}$,
such that it is still invariant under M\"{o}bius transformations?
\end{question}

Say, for closed regions in $\partial\mathbb{H}^{2m+1}$ with piecewise smooth boundary, and if so, how?%
\footnote{For $m=1$, for any closed region $U$ in $\partial\mathbb{H}^3$ 
with piecewise smooth boundary, a potential definition of the volume $V_{\infty,2}(U)$ 
(that is still finitely additive) is as follows.
First partition $U$ into some simplicial regions with piecewise smooth boundaries,
for each region define the volume as $\alpha+\beta+\gamma-\pi$ 
where $\alpha$, $\beta$ and $\gamma$ are the dihedral angles,
then sum them up to define $V_{\infty,2}(U)$.
It can be shown that $V_{\infty,2}(U)$ is well defined and independent of the partition
(see also Corollary~\ref{corollary_volume_infinity_two}).
By the definition, $V_{\infty,2}(U)$ is not only invariant under M\"{o}bius transformations,
but also invariant under any conformal mappings on the \emph{closed} region $U$.
It should not be confused with the Riemman mapping theorem (for mapping to an open disk)
whose subjects are simply connected \emph{open} regions where the mapping may not be necessarily conformal
on the boundary.
}
For $n$ even, we do not have notions of length or volume for $\partial\mathbb{H}^n$:
this is because it is possible for $V_n(P)$ to be non-zero while 
$G=P\cap\partial\mathbb{H}^n$ is the empty set
(e.g., when $P_{+}$ is a convex polytope in $\mathbb{H}^n$),
thus we cannot simply assign $V_n(P)$ to $G$ to obtain a well defined function of $G$.
For $n=2$, another way to see this is that any interval on $\partial\mathbb{H}^2$ is a half-space, 
so we cannot define a non-trivial length on $\partial\mathbb{H}^2$.

\begin{remark}
\label{remark_missing_face}
We remark that unlike the polytopes in $M^n$ or $\mathbb{DH}^n$,
a polytope $G$ in $\partial\mathbb{H}^n$ may not be formed 
by a \emph{unique} minimal set of half-spaces. 
For example in $\partial\mathbb{H}^3$ (a 2-sphere),
the intersection of \emph{three} properly chosen 
half-spaces may contain two \emph{simplicial} components $G$ and $F$
(like in $\mathbb{DH}^2$ the intersection of three half-spaces 
may contain a pair of simplices in $\mathbb{H}^2$ and $\mathbb{H}^2_{-}$ respectively).
But in order to make $G$ itself a polytope in $\partial\mathbb{H}^3$, 
a \emph{fourth} half-space, 
whose choice is not unique, has to be added to separate $G$ and $F$. 
So while $G$ only has three visible sides, it is formed by a set of at least four half-spaces
in $\partial\mathbb{H}^3$ whose choice is not unique. See Figure~\ref{figure_fourth_sphere}.
\end{remark}

\begin{figure}[h]
\centering
  \includegraphics[width=0.5\textwidth]{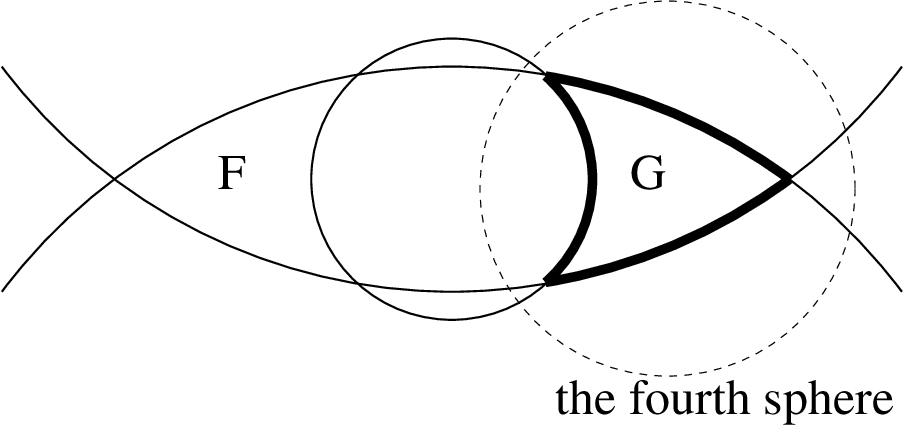}
\caption{A polytope $G$ in $\partial\mathbb{H}^n$ not formed by a unique minimal set of half-spaces}
\label{figure_fourth_sphere}
\end{figure}

For $n\ge 3$, for two half-spaces in $\partial\mathbb{H}^n$
whose boundaries intersect, an intersection \emph{angle} is well defined,
which is the same intersection angle between the two corresponding
half-spaces in $\mathbb{DH}^n$.
For a polytope $P$ in $\mathbb{DH}^{2m+1}$,
we already prove that $V_{2m+1}(P)$ satisfies the \SDF{} for $\mathbb{DH}^{2m+1}$.
For all codimension 2 faces of $P$, as their volumes and the dihedral angles
can also be passed through to $\partial\mathbb{H}^{2m+1}$ 
with a restriction to $\partial\mathbb{H}^{2m+1}$,
so the \SDF{} for $\mathbb{DH}^{2m+1}$
can also be passed through to $\partial\mathbb{H}^{2m+1}$
adjusted by a set of factors $c_{2m}$ in (\ref{equation_volume_boundary}).
By Theorem~\ref{theorem_schlafli}, we have the
the following new version of \SDF{} for $\partial\mathbb{H}^{2m+1}$.

\begin{theorem}
\label{theorem_schlafli_boundary_infinity_new}
For $m\ge 1$, let $P$ be a polytope in $\mathbb{DH}^{2m+1}$ and $G=P\cap\partial\mathbb{H}^{2m+1}$.
If $P$ does not contain any ideal vertices, then 
\begin{equation}
\label{equation_schlafli_boundary}
dV_{\infty,2m}(G)=\frac{1}{2m-1}\sum_{F}V_{\infty,2m-2}(F)\,d\theta_F,
\end{equation}
where the sum is taken over all codimension 2 faces $F$ of $G$.
For $2m-2=0$, $V_{\infty,0}(F)$ is the number of points in $F$.
\end{theorem}

We immediately have the following result for dimension two.

\begin{corollary}
\label{corollary_volume_infinity_two}
Let $G$ be a polytope in $\partial\mathbb{H}^3$ homeomorphic to a closed disk,
and have $k$ sides with dihedral angles $\theta_i$ between consecutive sides, then
\begin{equation}
\label{equation_volume_infinity_two}
V_{\infty,2}(G)=\sum_i\theta_i - (k-2)\pi.
\end{equation}
\end{corollary}

\begin{proof}
In (\ref{equation_schlafli_boundary}), for $m=1$ by integrating both sides we have
$V_{\infty,2}(G)=\sum_i\theta_i+c$.
By a limiting case with $\theta_1=\theta_k=0$ and $\theta_i=\pi$ for $2\le i\le k-1$, 
where the volume is 0, the constant term is determined to be  $-(k-2)\pi$.
This finishes the proof.
\end{proof}

We fix a standard unit sphere $\mathbb{S}^{2m}$ as a model for $\partial\mathbb{H}^{2m+1}$.
Let $G$ be a \emph{spherical} convex polytope in $\mathbb{S}^{2m}$
and $V_{2m}(G)$ be the standard spherical volume.
As $V_{2m}(G)$ satisfies a \SDF{} with $\kappa=1$,
then by Theorem~\ref{theorem_schlafli_boundary_infinity_new} and an induction on $m$, 
we have
\begin{equation}
\label{equation_boundary_infinity_spherical}
V_{\infty,2m}(G)=V_{2m}(G).
\end{equation}
But somewhat surprisingly, this identity does not hold in general when $G$ is not a spherical convex polytope,
as $V_{\infty,2m}(G)$ is invariant under M\"{o}bius transformations while $V_{2m}(G)$ is not.

\begin{example}
\label{example_s_2}
For $m=1$, let $G$ be formed by three non-intersecting segments of small circles on $\mathbb{S}^2$
with interior angles $\alpha$, $\beta$, and $\gamma$. 
By Corollary~\ref{corollary_volume_infinity_two}
we have $V_{\infty,2}(G)=\alpha+\beta+\gamma-\pi$. 
When the three sides of $G$ form a small circle on $\mathbb{S}^2$
with $\alpha=\beta=\gamma=\pi$, then $V_{\infty,2}(G)=2\pi$, 
but $V_2(G)$ can take any value between 0 and $4\pi$.
Thus $V_{\infty,2}(G)\ne V_2(G)$.
\end{example}

Recall that $\mathbb{S}^{2m}$, $\mathbb{DH}^{2m}$, and $\mathbb{R}^{2m}$ (plus a $\infty$)
are all naturally endowed with the same conformal structure as of $\partial\mathbb{H}^{2m+1}$,
so for a polytope $G$ in $M^{2m}$ or $\mathbb{DH}^{2m}$,
we can also assign a value $V_{\infty,2m}(G)$ to $G$.
We show that $\partial\mathbb{H}^{2m+1}$ exhibits geometric properties of the spherical,
(double) hyperbolic, as well as Euclidean spaces at the same time, in the following sense.
For a polytope $G$ in $\mathbb{DH}^{2m}$,
as $V_{2m}(G)$ satisfies a \SDF{} with $\kappa=-1$,
then by Theorem~\ref{theorem_schlafli_boundary_infinity_new} and an induction on $m$, 
we have $V_{\infty,2m}(G)=(-1)^m V_{2m}(G)$.
If the upper portion of $G$  (denote by $G_{+}$) is also 
a convex polytope in $\mathbb{H}^{2m}$, then
$V_{\infty,2m}(G_{+})=\frac{1}{2}V_{\infty,2m}(G)$ and $V_{2m}(G_{+})=\frac{1}{2}V_{2m}(G)$,
thus $V_{\infty,2m}(G_{+})=(-1)^m V_{2m}(G_{+})$;
but it is not so if $G_{+}$ is an unbounded polytope in $\mathbb{H}^{2m}$,
since $V_{\infty,2m}(G_{+})=\frac{1}{2}V_{\infty,2m}(G)$ still holds but $V_{2m}(G_{+})$ does not exist.
To summarize, we have the following result generalizing 
(\ref{equation_boundary_infinity_spherical}).

\begin{theorem}
\label{theorem_polytope_volume_infinity}
Let $G$ be a convex polytope in $M^{2m}$ or a polytope in $\mathbb{DH}^{2m}$
of constant curvature $\kappa$, then
$V_{\infty,2m}(G)=\kappa^m V_{2m}(G)$.
\end{theorem}

\begin{proof}
The only case left is to show that when $G$ is a convex polytope in $\mathbb{R}^{2m}$ (with $\kappa=0$),
we have $V_{\infty,2m}(G)=0$. Without loss of generality, we assume $G$ is a simplex.
For $m=1$, by (\ref{equation_volume_infinity_two}) we have $V_{\infty,2}(G)=0$.
For $m>1$, by induction the right side of (\ref{equation_schlafli_boundary}) is 0,
so $dV_{\infty,2m}(G)=0$. As $G$ can continuously deform to a degenerate Euclidean simplex 
with zero volume, thus $V_{\infty,2m}(G)=0$. 
\end{proof}

Similar to Example~\ref{example_s_2},
Theorem~\ref{theorem_polytope_volume_infinity} does not hold in general 
when $G$ is not a convex polytope in $M^{2m}$ or a polytope in $\mathbb{DH}^{2m}$.
In the (double) hyperbolic case, this is so because $V_{\infty,2m}(G)$ 
is invariant under M\"{o}bius transformations of $\mathbb{DH}^{2m}$ while $V_{2m}(G)$ is not
(when the M\"{o}bius transformation does \emph{not} preserve the boundary at infinity $\partial\mathbb{H}^{2m}$).
In the Euclidean case, if $G$ is a ball in $\mathbb{R}^{2m}$,
then $V_{\infty,2m}(G)$ has a fixed non-zero volume (see Remark~\ref{remark_boundary_ball_volume})
while $\kappa^m V_{2m}(G)$ is 0.
We also remark that if $G$ is an unbounded polytope in $\mathbb{R}^{2m}$,
we do not have $V_{\infty,2m}(G)=0$ either.


\noindent
{\bf Acknowledgements:}
I would like to thank Wei Luo and a referee for carefully reading the manuscript
and making many helpful suggestions.

{\footnotesize
\bibliographystyle{abbrv}  
\bibliography{double_hyperbolic_arxiv}   

%
%
}

\end{document}

%% file: fig_area_element.pspdftex
\begin{picture}(0,0)%
\includegraphics{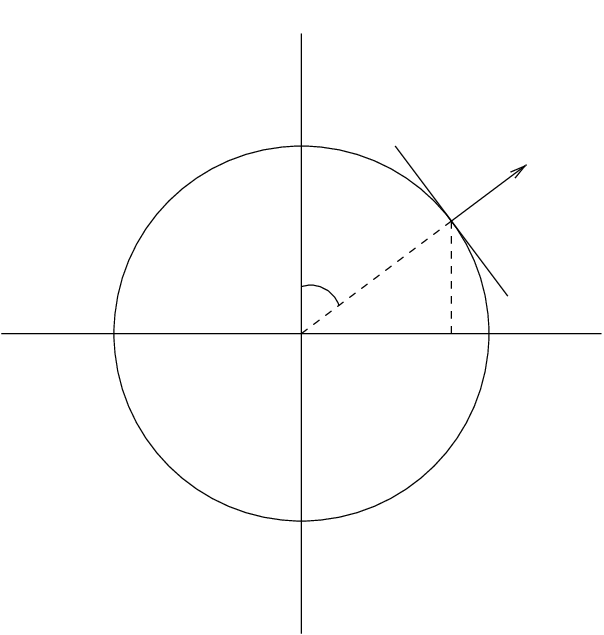}%
\end{picture}%
\setlength{\unitlength}{3947sp}%
\begingroup\makeatletter\ifx\SetFigFont\undefined%
\gdef\SetFigFont#1#2#3#4#5{%
  \reset@font\fontsize{#1}{#2pt}%
  \fontfamily{#3}\fontseries{#4}\fontshape{#5}%
  \selectfont}%
\fi\endgroup%
\begin{picture}(4824,5076)(3589,-6373)
\put(6151,-3511){\makebox(0,0)[lb]{\smash{{\SetFigFont{20}{24.0}{\rmdefault}{\mddefault}{\updefault}{\color[rgb]{0,0,0}$\theta$}%
}}}}
\put(6151,-1636){\makebox(0,0)[lb]{\smash{{\SetFigFont{20}{24.0}{\rmdefault}{\mddefault}{\updefault}{\color[rgb]{0,0,0}$x_0$}%
}}}}
\put(7501,-3211){\makebox(0,0)[lb]{\smash{{\SetFigFont{20}{24.0}{\rmdefault}{\mddefault}{\updefault}{\color[rgb]{0,0,0}$dA_r$}%
}}}}
\put(6676,-3286){\makebox(0,0)[lb]{\smash{{\SetFigFont{20}{24.0}{\rmdefault}{\mddefault}{\updefault}{\color[rgb]{0,0,0}$r$}%
}}}}
\put(6901,-3736){\makebox(0,0)[lb]{\smash{{\SetFigFont{20}{24.0}{\rmdefault}{\mddefault}{\updefault}{\color[rgb]{0,0,0}$x_0$}%
}}}}
\end{picture}%

%% file: fig_projection.pspdftex
\begin{picture}(0,0)%
\includegraphics{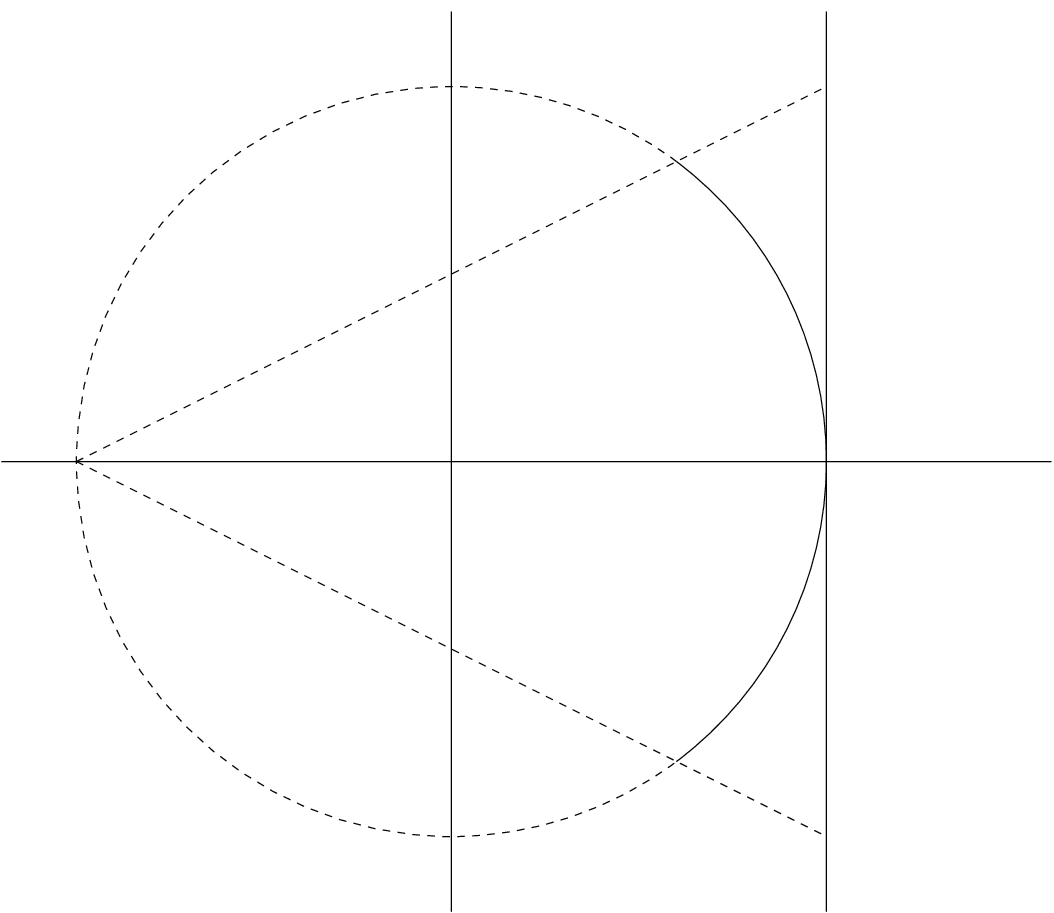}%
\end{picture}%
\setlength{\unitlength}{3947sp}%
\begingroup\makeatletter\ifx\SetFigFont\undefined%
\gdef\SetFigFont#1#2#3#4#5{%
  \reset@font\fontsize{#1}{#2pt}%
  \fontfamily{#3}\fontseries{#4}\fontshape{#5}%
  \selectfont}%
\fi\endgroup%
\begin{picture}(8424,7296)(2389,-7573)
\put(9226,-4486){\makebox(0,0)[lb]{\smash{{\SetFigFont{34}{40.8}{\rmdefault}{\mddefault}{\updefault}{\color[rgb]{0,0,0}1}%
}}}}
\put(6226,-736){\makebox(0,0)[lb]{\smash{{\SetFigFont{34}{40.8}{\rmdefault}{\mddefault}{\updefault}{\color[rgb]{0,0,0}$x_0$}%
}}}}
\put(6151,-4486){\makebox(0,0)[lb]{\smash{{\SetFigFont{34}{40.8}{\rmdefault}{\mddefault}{\updefault}{\color[rgb]{0,0,0}0}%
}}}}
\put(10426,-4411){\makebox(0,0)[lb]{\smash{{\SetFigFont{34}{40.8}{\rmdefault}{\mddefault}{\updefault}{\color[rgb]{0,0,0}$x_n$}%
}}}}
\put(2551,-4486){\makebox(0,0)[lb]{\smash{{\SetFigFont{34}{40.8}{\rmdefault}{\mddefault}{\updefault}{\color[rgb]{0,0,0}-1}%
}}}}
\end{picture}%

%% file: fig_angle.pspdftex
\begin{picture}(0,0)%
\includegraphics{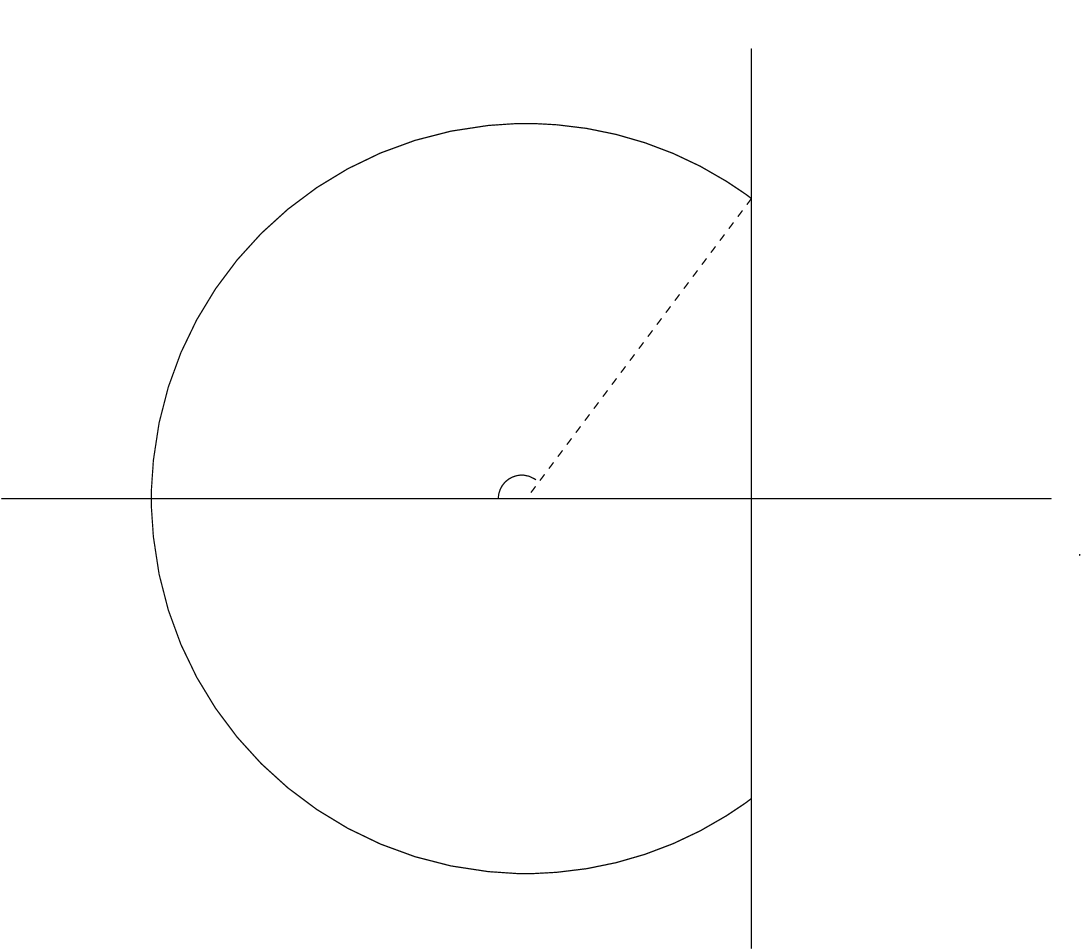}%
\end{picture}%
\setlength{\unitlength}{3947sp}%
\begingroup\makeatletter\ifx\SetFigFont\undefined%
\gdef\SetFigFont#1#2#3#4#5{%
  \reset@font\fontsize{#1}{#2pt}%
  \fontfamily{#3}\fontseries{#4}\fontshape{#5}%
  \selectfont}%
\fi\endgroup%
\begin{picture}(8649,7596)(2389,-7573)
\put(8476,-1636){\makebox(0,0)[lb]{\smash{{\SetFigFont{34}{40.8}{\rmdefault}{\mddefault}{\updefault}{\color[rgb]{0,0,0}$F$}%
}}}}
\put(4576,-4936){\makebox(0,0)[lb]{\smash{{\SetFigFont{34}{40.8}{\rmdefault}{\mddefault}{\updefault}{\color[rgb]{0,0,0}$P_t$}%
}}}}
\put(7201,-2611){\makebox(0,0)[lb]{\smash{{\SetFigFont{34}{40.8}{\rmdefault}{\mddefault}{\updefault}{\color[rgb]{0,0,0}$r$}%
}}}}
\put(8551,-436){\makebox(0,0)[lb]{\smash{{\SetFigFont{34}{40.8}{\rmdefault}{\mddefault}{\updefault}{\color[rgb]{0,0,0}$x_0$}%
}}}}
\put(10501,-4411){\makebox(0,0)[lb]{\smash{{\SetFigFont{34}{40.8}{\rmdefault}{\mddefault}{\updefault}{\color[rgb]{0,0,0}$x_{n-1}$}%
}}}}
\put(6451,-4336){\makebox(0,0)[lb]{\smash{{\SetFigFont{34}{40.8}{\rmdefault}{\mddefault}{\updefault}{\color[rgb]{0,0,0}$c$}%
}}}}
\put(8626,-4411){\makebox(0,0)[lb]{\smash{{\SetFigFont{34}{40.8}{\rmdefault}{\mddefault}{\updefault}{\color[rgb]{0,0,0}$t$}%
}}}}
\put(8551,-5386){\makebox(0,0)[lb]{\smash{{\SetFigFont{34}{40.8}{\rmdefault}{\mddefault}{\updefault}{\color[rgb]{0,0,0}$E_t$}%
}}}}
\put(4501,-1261){\makebox(0,0)[lb]{\smash{{\SetFigFont{34}{40.8}{\rmdefault}{\mddefault}{\updefault}{\color[rgb]{0,0,0}$E'$}%
}}}}
\put(8476,-2686){\makebox(0,0)[lb]{\smash{{\SetFigFont{34}{40.8}{\rmdefault}{\mddefault}{\updefault}{\color[rgb]{0,0,0}$r_F$}%
}}}}
\put(6001,-3661){\makebox(0,0)[lb]{\smash{{\SetFigFont{34}{40.8}{\rmdefault}{\mddefault}{\updefault}{\color[rgb]{0,0,0}$\theta_F$}%
}}}}
\end{picture}%

%% file: fig_polar.pspdftex
\begin{picture}(0,0)%
\includegraphics{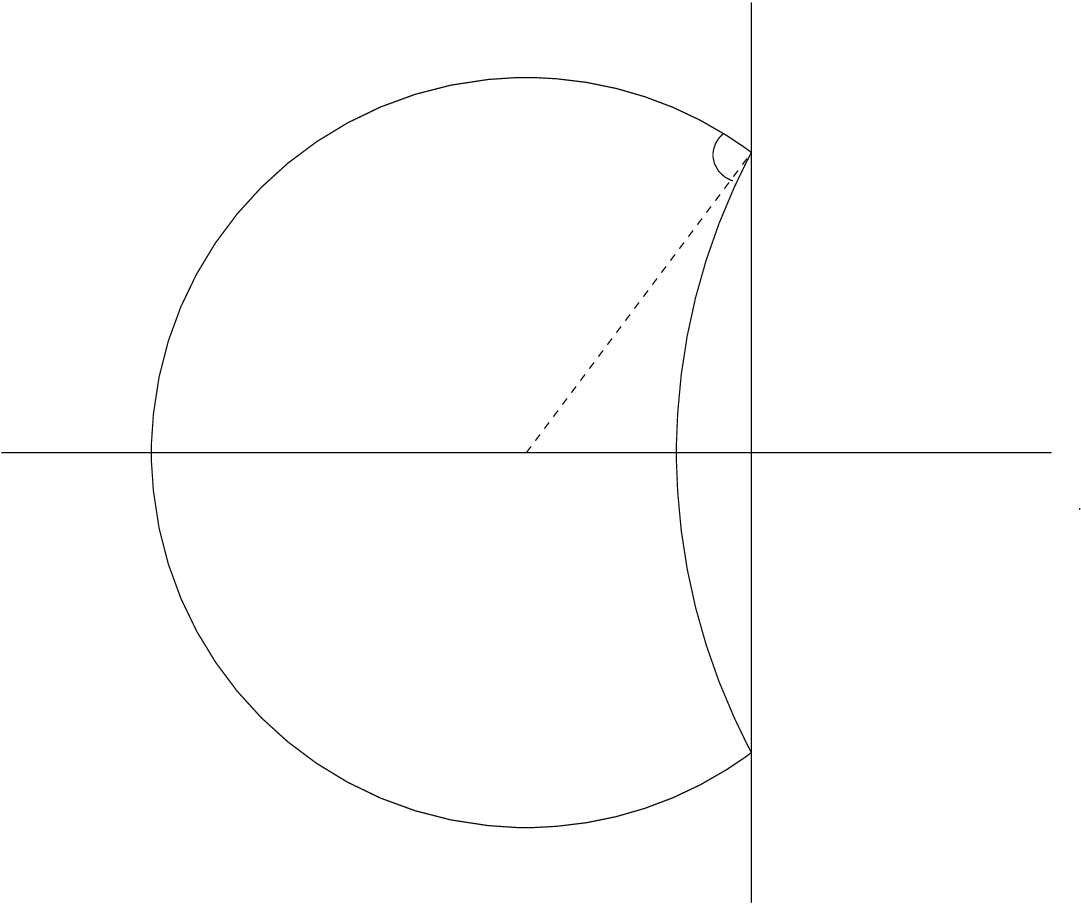}%
\end{picture}%
\setlength{\unitlength}{3947sp}%
\begingroup\makeatletter\ifx\SetFigFont\undefined%
\gdef\SetFigFont#1#2#3#4#5{%
  \reset@font\fontsize{#1}{#2pt}%
  \fontfamily{#3}\fontseries{#4}\fontshape{#5}%
  \selectfont}%
\fi\endgroup%
\begin{picture}(8649,7224)(2389,-7573)
\put(7351,-1861){\makebox(0,0)[lb]{\smash{{\SetFigFont{34}{40.8}{\rmdefault}{\mddefault}{\updefault}{\color[rgb]{0,0,0}$\theta_{F_i}$}%
}}}}
\put(4501,-1261){\makebox(0,0)[lb]{\smash{{\SetFigFont{34}{40.8}{\rmdefault}{\mddefault}{\updefault}{\color[rgb]{0,0,0}$E$}%
}}}}
\put(8551,-5386){\makebox(0,0)[lb]{\smash{{\SetFigFont{34}{40.8}{\rmdefault}{\mddefault}{\updefault}{\color[rgb]{0,0,0}$H_i$}%
}}}}
\put(8551,-1636){\makebox(0,0)[lb]{\smash{{\SetFigFont{34}{40.8}{\rmdefault}{\mddefault}{\updefault}{\color[rgb]{0,0,0}$F_i$}%
}}}}
\put(6526,-4486){\makebox(0,0)[lb]{\smash{{\SetFigFont{34}{40.8}{\rmdefault}{\mddefault}{\updefault}{\color[rgb]{0,0,0}$O$}%
}}}}
\put(4576,-4936){\makebox(0,0)[lb]{\smash{{\SetFigFont{34}{40.8}{\rmdefault}{\mddefault}{\updefault}{\color[rgb]{0,0,0}$P'_r$}%
}}}}
\put(7276,-4936){\makebox(0,0)[lb]{\smash{{\SetFigFont{34}{40.8}{\rmdefault}{\mddefault}{\updefault}{\color[rgb]{0,0,0}$E_i$}%
}}}}
\put(6826,-2986){\makebox(0,0)[lb]{\smash{{\SetFigFont{34}{40.8}{\rmdefault}{\mddefault}{\updefault}{\color[rgb]{0,0,0}$r_0$}%
}}}}
\end{picture}%

%% file: double_hyperbolic_arxiv.bbl
\begin{thebibliography}{10}

\bibitem{Alexander:Lipschitzian}
R.~Alexander.
\newblock Lipschitzian mappings and total mean curvature of polyhedral
  surfaces. {I.}
\newblock {\em Trans. Am. Math. Soc.}, 288(2):661--678, 1985.

\bibitem{Cannon:hyperbolic}
J.~W. Cannon, W.~J. Floyd, R.~Kenyon, and W.~R. Parry.
\newblock Hyperbolic geometry.
\newblock {\em Flavors of geometry}, 31:59--115, 1997.

\bibitem{ChoKim}
Y.~Cho and H.~Kim.
\newblock The analytic continuation of hyperbolic space.
\newblock {\em Geom. Dedicata}, 161(1):129--155, 2012.

\bibitem{HaagerupMunkholm}
U.~Haagerup and H.~J. Munkholm.
\newblock Simplices of maximal volume in hyperbolic $n$-space.
\newblock {\em Acta Math.}, 147:1--11, 1981.

\bibitem{Luo:continuity}
F.~Luo.
\newblock Continuity of the volume of simplices in classical geometry.
\newblock {\em Commun. Contemp. Math.}, 8(3):411--431, 2006.

\bibitem{Milnor:hyperbolic}
J.~Milnor.
\newblock Hyperbolic geometry: The first 150 years.
\newblock {\em Bull. Amer. Math. Soc. (N.S.)}, 6(1):9--24, 1982.

\bibitem{Milnor:Schlafli}
J.~Milnor.
\newblock The {S}chl{\"a}fli differential equality.
\newblock In {\em Collected Papers, vol. 1}. Publish or Perish, New York, 1994.

\bibitem{Rivin:volumes}
I.~Rivin.
\newblock Volumes of degenerating polyhedra -- on a conjecture of {J. W.
  Milnor}.
\newblock {\em Geom. Dedicata}, 131(1):73--85, 2008.

\bibitem{RivinSchlenker}
I.~Rivin and J.-M. Schlenker.
\newblock The {S}chl{\"a}fli formula in {Einstein} manifolds with boundary.
\newblock {\em Electron. Res. Announc. Am. Math. Soc.}, 5:18--23, 1999.

\bibitem{Schlenker:cross}
J.-M. Schlenker.
\newblock M{\'e}triques sur les poly{\`e}dres hyperboliques convexes.
\newblock {\em J. Diff. Geom.}, 48:323--405, 1998.

\bibitem{Suarez:deSitter}
E.~Su{\'a}rez-Peir{\'o}.
\newblock A {S}chl{\"a}fli differential formula for simplices in
  semi-{Riemannian} hyperquadrics, {Gauss-Bonnet} formulas for simplices in the
  de {Sitter} sphere and the dual volume of a hyperbolic simplex.
\newblock {\em Pac. J. Math.}, 194(1):229--255, 2000.

\bibitem{Zhang:rigidity}
L.~Zhang.
\newblock Rigidity and volume preserving deformation on degenerate simplices.
\newblock {\em Discrete Comput. Geom.}, 60(4):909--937, 2018.

\bibitem{Ziegler:polytopes}
G.~M. Ziegler.
\newblock {\em Lectures on Polytopes}, volume 152 of {\em Graduate Texts in
  Mathematics}.
\newblock Springer, New York, 1995.

\end{thebibliography}
